\documentclass[12pt,twoside]{article}
\usepackage{amsmath,amsthm,amssymb,amscd,ascmac, amsfonts}
\usepackage{mathrsfs}
\usepackage{stmaryrd}

\numberwithin{equation}{section}
\usepackage[dvips]{graphicx,color,psfrag}

\makeatletter
\newcommand{\figcaption}[1]{\def\@captype{figure}\caption{#1}}
\newcommand{\tblcaption}[1]{\def\@captype{table}\caption{#1}}
\makeatother

\newcommand{\Det}{\mathrm{Det}\,}

\makeatletter

\ifcase\@ptsize
\def\rpkern{\mathchoice{\kern-1.45em}{\kern-1.11em}{}{}}%
\def\grpkern{\mathchoice{\kern-1.013em}{\kern-0.825em}{}{}}%
\or
\def\rpkern{\mathchoice{\kern-1.44em}{\kern-1.11em}{}{}}%
\def\grpkern{\mathchoice{\kern-1.00em}{\kern-0.81em}{}{}}%
\or
\def\rpkern{\mathchoice{\kern-1.472em}{\kern-1.14em}{}{}}%
\def\grpkern{\mathchoice{\kern-1.00em}{\kern-0.815em}{}{}}%
\fi

\def\minibullet{\mathchoice%
{\raise0.2ex\hbox{$\scriptstyle\bullet$}}%
{\raise0.26ex\hbox{$\scriptscriptstyle\bullet$}}{}{}}
\def\butabullet{\mathchoice%
{\raise0.8ex\hbox{$\scriptstyle\bullet$}{\kern-0.365em}%
\lower0.4ex\hbox{$\scriptstyle\bullet$}}%
{\raise0.75ex\hbox{$\scriptscriptstyle\bullet$}{\kern-0.335em}%
\lower0.25ex\hbox{$\scriptscriptstyle\bullet$}}{}{}}

\def\customprod#1#2%
{\operatorname{\coprod\kern#2{#1}\kern#2\prod}\displaylimits}

\renewcommand{\d}{\delta}

\renewcommand{\Re}{\mathrm{Re}\,}

\newcommand{\Aut}{\mathop{\mathrm{Aut}}\,}
\newcommand{\tr}{\mathop{\mathrm{tr}}\,}
\renewcommand{\det}{\mathop{\mathrm{det}}\,}

\newenvironment{msc}{%
\smallbreak
\noindent \textit{Mathematics Subject Classification 2010\,:\,}}

\newenvironment{keywords}{%
\noindent\text{Key words and Phrases\,:\,}}

\newtheorem*{multitheorem}{\variable@name}

\theoremstyle{definition}
\newcommand{\variable@name}{Theorem}
\newtheorem*{multiproclaim}{\variable@name}

\theoremstyle{plain}
\newtheorem{thm}{Theorem}[section]
\newtheorem{prop}[thm]{Proposition}
\newtheorem{lem}[thm]{Lemma}
\newtheorem{cor}[thm]{Corollary}

\newtheorem{conj}[thm]{Conjecture}

\theoremstyle{definition}
\newtheorem{dfn}[thm]{Definition}
\newtheorem{rmk}[thm]{Remark}

\begin{document}
\title{Multivariate Meixner, Charlier and Krawtchouk polynomials}
\author{Genki Shibukawa
\thanks{
This work was supported by Grant-in-Aid for JSPS Fellows (Number 12J04930). 
}}
\date{\empty}





\maketitle


\begin{abstract}
We introduce some multivariate analogues of Meixner, Charlier and Krawtchouk polynomials, 
and establish their main properties, that is, duality, degenerate limits, generating functions, orthogonality relations, difference equations, recurrence formulas and determinant expressions. 
A particularly important and interesting result is that ``the generating function of the generating functions'' for the Meixner polynomials coincides with the generating function of the Laguerre polynomials, which has previously not been known even for the one variable case.  
Actually, main properties for the multivariate Meixner, Charlier and Krawtchouk polynomials are derived from some properties of the multivariate Laguerre polynomials by using this key result. 
 \\
\end{abstract}


\section{Introduction}
The standard Meixner, Charlier and Krawtchouk polynomials of a single discrete variable are defined by 
\begin{align}
M_{m}(x;\alpha ,c)&:={_{2}F_1}\left(\begin{matrix}-m,-x\\ \alpha \end{matrix};1-\frac{1}{c}\right)
=\sum_{k=0}^{m}\frac{k!}{(\alpha)_{k}}\binom{m}{k}\binom{x}{k}\left(1-\frac{1}{c}\right)^{k}, \nonumber \\
C_{m}(x;a)&:={_{2}F_0}\left(\begin{matrix}-m,-x\\ {-} \end{matrix};-\frac{1}{a}\right)
=\sum_{k=0}^{m}k!\binom{m}{k}\binom{x}{k}\left(-\frac{1}{a}\right)^{k}, \nonumber \\
K_{m}(x;p,N)&:={_{2}F_1}\left(\begin{matrix}-m,-x\\ -N \end{matrix};\frac{1}{p}\right)
=\sum_{k=0}^{m}\frac{k!}{(-N)_{k}}\binom{m}{k}\binom{x}{k}\left(\frac{1}{p}\right)^{k}, \nonumber
\end{align}
respectively. 
These polynomials have been generalized to the multivariate case \cite{DG}, \cite{Gr1}, \cite{Gr2}, and \cite{Il}. 
Although these multivariate discrete orthogonal polynomials are written by the Aomoto-Gelfand hypergeometric series, 
we introduce other types of multivariate Meixner, Charlier and Krawtchouk polynomials in this article, which are defined by the generalized binomial coefficients. 
Moreover, we provide their fundamental properties, 
that is, duality, degenerate limits, generating functions, orthogonality relations, difference equations and recurrence formulas. 
The most basic result in these properties is Theorem\,\ref{thm:master generating fnc 1}, 
which states that ``the generating function of the generating functions'' for the multivariate Meixner polynomials 
$$
\sum_{\mathbf{x} \in \mathscr{P}}d_{\mathbf{x}}\frac{1}{\left(\frac{n}{r}\right)_{\mathbf{x}}}
\left\{\sum_{\mathbf{m} \in \mathscr{P}}d_{\mathbf{m}}\frac{(\alpha)_{\mathbf{m}}}{\left(\frac{n}{r}\right)_{\mathbf{m}}}M_{\mathbf{m}}(\mathbf{x};\alpha ,c)\Phi_{\mathbf{m}}(z)\right\} \Phi_{\mathbf{x}}(w)
$$
coincides with the generating function for the multivariate Laguerre polynomials
$$
\sum_{\mathbf{m} \in \mathscr{P}}e^{\tr{w}}L_{\mathbf{m}}^{\left(\alpha -\frac{n}{r}\right)}\left(\left(\frac{1}{c}-1\right)w\right)\Phi_{\mathbf{m}}(z).
$$
Even though this result has not been known even for one variable, many properties for our multivariate discrete special orthogonal polynomials follow from this  and the unitary picture (\ref{eq:unitary picture2}) according to analysis on the symmetric cones. 

Let us describe our scheme in the one variable case more precisely. 
We put $\alpha>1$, $(\alpha)_{m}:=\frac{\Gamma(\alpha+m)}{\Gamma(\alpha)}=\alpha(\alpha+1)\cdots(\alpha+m-1)$, $\binom{m}{k}=(-1)^{k}\frac{(-m)_{k}}{k!}$, $\mathcal{D}:=\{w \in \mathbb{C}\mid \vert w \vert<1\}$, $T:=\{z \in \mathbb{C}\mid \Re{z}>0\}$, $m$ is the Lebesgue measure on $\mathbb{C}$. 
Further, we introduce the following function spaces and their complete orthogonal bases.  
\\
{\bf{(1)}}\,\,$\psi_{m}^{(\alpha)}$\,;\,exponential multiplied by the Laguerre polynomials 
\begin{align}
L^{2}_{\alpha}(\mathbb{R}_{>0})&:=\{\psi:\mathbb{R}_{>0} \longrightarrow \mathbb{C} \mid \|\psi\|_{\alpha,\mathbb{R}_{>0}}^{2}<\infty\}, \nonumber \\
\|\psi\|_{\alpha,\mathbb{R}_{>0}}^{2}&:=\frac{2^{\alpha}}{\Gamma(\alpha)}\int_{0}^{\infty}|\psi(u)|^{2}u^{\alpha -1}\,du, \nonumber \\
\psi_{m}^{(\alpha)}(u)&:=e^{-u}L_{m}^{(\alpha-1)}(2u)
=\frac{(\alpha)_{m}}{m!}e^{-u}\sum_{k=0}^{m}(-1)^{k}\binom{m}{k}\frac{1}{(\alpha)_{k}}(2u)^{k}.\nonumber
\end{align}
\noindent
{\bf{(2)}}\,\,$F_{m}^{(\alpha)}$\,;\,Cayley transform of the polynomials
\begin{align}
\mathcal{H}^{2}_{\alpha}(T)&:=\{F:T \longrightarrow \mathbb{C} \mid F \text{ is analytic in $T$ and }\|F\|_{\alpha,T}^{2}<\infty\}, \nonumber \\
\|F\|_{\alpha,T}^{2}&:=\frac{\alpha -1}{4\pi}\int_{T}|F(z)|^{2}x^{\alpha -2}\,m(dz), \nonumber \\
F_{m}^{(\alpha)}(z)&:=\frac{(\alpha)_{m}}{m!}\left(\frac{1+z}{2}\right)^{-\alpha}\left(\frac{z-1}{z+1}\right)^{m}. \nonumber
\end{align}
{\bf{(3)}}\,\,$f_{m}^{(\alpha)}$\,;\,monomials \\
\begin{align}
\mathcal{H}^{2}_{\alpha}(\mathcal{D})&:=\{f:\mathcal{D} \longrightarrow \mathbb{C} \mid f \text{ is analytic in $\mathcal{D}$ and }\|f\|_{\alpha,\mathcal{D}}^{2}<\infty\}, \nonumber \\
\|f\|_{\alpha,\mathcal{D}}^{2}&:=\frac{\alpha -1}{\pi}\int_{\mathcal{D}}|f(w)|^{2}(1-|w|^{2})^{\alpha -2}\,m(dw), \nonumber \\
f_{m}^{(\alpha)}(w)&:=\frac{(\alpha)_{m}}{m!}w^{m}.\nonumber 
\end{align}
We remark that 
\begin{equation}
\|\psi_{m}^{(\alpha)}\|_{\alpha,\mathbb{R}_{>0}}^{2}=\|F_{m}^{(\alpha)}\|_{\alpha,T}^{2}=\|f_{m}^{(\alpha)}\|_{\alpha,\mathcal{D}}^{2}
=\frac{(\alpha)_{m}}{m!}. \nonumber
\end{equation}
Furthermore, the following unitary isomorphisms are known. \\
\underline{Modified Laplace transform}
\begin{align}
\mathcal{L}_{\alpha}:L^{2}_{\alpha}(\mathbb{R}_{>0}) \xrightarrow{\simeq}  \mathcal{H}^{2}_{\alpha}(T),
\,\,\,(\mathcal{L}_{\alpha}\psi)(z):=\frac{2^{\alpha}}{\Gamma(\alpha)}\int_{0}^{\infty}e^{-zu}u^{\alpha-1}\psi(u)\,du.\nonumber
\end{align}
\underline{Modified Cayley transform}
\begin{align}
C_{\alpha}^{-1}:\mathcal{H}^{2}_{\alpha}(T) \xrightarrow{\simeq} \mathcal{H}^{2}_{\alpha}(\mathcal{D}),\,\,\,(C_{\alpha}^{-1}F)(w):=(1-w)^{-\alpha}F\left(\frac{1+w}{1-w}\right). \nonumber
\end{align}
To summarize, we obtain the following picture given by the unitary transformations. 
\begin{align}
\label{eq:unitary picture}
\begin{array}{cccccc}
L^{2}_{\alpha}(\mathbb{R}_{>0}) & \xrightarrow[\mathcal{L}_{\alpha}]{\simeq} & \mathcal{H}^{2}_{\alpha}(T) & \xrightarrow[C_{\alpha}^{-1}]{\simeq} & \mathcal{H}^{2}_{\alpha}(\mathcal{D}). & \\
\rotatebox{90}{$\in$}  & & \rotatebox{90}{$\in$} & & \rotatebox{90}{$\in$} & \\
\psi_{m}^{(\alpha)}  & \longmapsto & F_{m}^{(\alpha)}& \longmapsto & f_{m}^{(\alpha)} &  \\
{\bf{(1)}} &  & {\bf{(2)}} &  & {\bf{(3)}} & 
\end{array}
\end{align}

On the other hand, by elementary calculation, we have
\begin{align}
\label{eq:key lem prot}
e^{-\frac{1+c}{1-c}u}\sum_{x \geq 0}\frac{1}{x!}\left(\frac{2c}{1-c}\right)^{x}
\left\{\sum_{m \geq 0}\frac{(\alpha)_{m}}{m!}M_{m}(x;\alpha ,c)z^{m}\right\}u^{x}
&=\sum_{m\geq 0}\psi_{m}^{(\alpha)}(u)z^{m} \nonumber \\
&=(1-z)^{-\alpha}e^{-u\frac{1+z}{1-z}}. 
\end{align}
It is interesting to note that there is a correspondence between Laguerre and Meixner polynomials. 
The former orthogonality is defined by the integral on $\mathbb{R}_{\geq 0}$ and the latter is defined by the summation on non negative integers.

From (\ref{eq:unitary picture}) and (\ref{eq:key lem prot}), for the Meixner polynomials, 
we derive {\bf{(a)}}\,generating function, {\bf{(b)}}\,orthogonality relation, {\bf{(c)}}\,difference equation and recurrence formula as follows.

\noindent
{\bf{(a)}}\,By comparing the coefficients of $u$ on the first equality of (\ref{eq:key lem prot}), 
$$
(1-z)^{-\alpha}\left(\frac{1-\frac{1}{c}z}{1-z}\right)^{x}=\sum_{m \geq 0}\frac{(\alpha)_{m}}{m!}M_{m}(x;\alpha ,c)z^{m}.
$$
{\bf{(b)}}\,By applying the unitary transformations $C_{\alpha}^{-1}\circ \mathcal{L}_{\alpha}$ in (\ref{eq:unitary picture}) to (\ref{eq:key lem prot}), we have
\begin{align}
\label{eq:key lem prot2}
(1-c)^{\alpha}\sum_{x\geq 0}\frac{(\alpha)_{x}}{x!}c^{x}\left\{\sum_{m \geq 0}\frac{(\alpha)_{m}}{m!}M_{m}(x;\alpha ,c)z^{m}\right\}
(1-cw)^{-\alpha}\left(\frac{1-w}{1-cw}\right)^{x}
\!\!&=\sum_{m \geq 0}\frac{(\alpha)_{m}}{m!}w^{m}z^{m} \nonumber \\
&=(1-wz)^{-\alpha}. 
\end{align}
We remark that the generating functions of the Meixner polynomials appear in the top left hand side of (\ref{eq:key lem prot2}). 
Hence, by using the generating functions of the Meixner polynomials and comparing the coefficients of $w$ and $z$ in (\ref{eq:key lem prot2}), 
we have the orthogonal relation for the Meixner polynomials. 
$$
\sum_{x\geq 0}\frac{(\alpha)_{x}}{x!}c^{x}M_{m}(x;\alpha ,c)M_{n}(x;\alpha ,c)
=\frac{c^{-m}}{(1-c)^{\alpha}}\frac{m!}{(\alpha)_{m}}\delta_{m,n}.
$$
{\bf{(c)}}\,We recall the differential operator $D_{\alpha}^{(1)}=-u\partial_{u}^{2}-\alpha\partial_{u}+u-\alpha$ which satisfies $D_{\alpha}^{(1)}\psi_{m}^{(\alpha)}(u)=2m\psi_{m}^{(\alpha)}(u)$. 
Therefore, by applying $\frac{c-1}{2}e^{\frac{1+c}{1-c}u}D_{\alpha}^{(1)}$ to (\ref{eq:key lem prot2}) and comparing the coefficients of $z$ and $u$, 
we obtain the following difference equation which is equivalent to a recurrence formula.  
\begin{align}
(c-1)mM_{m}(x;\alpha,c)&=(x+\alpha)cM_{m}(x+1;\alpha,c) \nonumber \\
 &\quad -(x+(x+\alpha)c)M_{m}(x;\alpha,c) \nonumber \\
 &\quad +xM_{m}(x-1;\alpha,c).\nonumber
\end{align}

The purpose of this article is to provide a multivariate analogue of this scheme which has previously not been known even for the one variable case. 
Let us now describe the content in this paper. 
The basic definitions and fundamental properties of Jordan algebras and symmetric cones, 
and lemmas for analysis on symmetric cones and tube domains have been presented in the first subsection of Section\,2, so that they can be referred to later. 
The next subsection presents a compilation of basic facts for the multivariate Laguerre polynomials and their unitary picture. 
Section\,3 which is the main part of this papers provides a multivariate analogue of the above results for Meixner, Charlier and Krawtchouk polynomials. 
Finally, in Section\,4, we present a conjecture and some problems for a further generalization of the multivariate Meixner, Charlier and Krawtchouk polynomials.

\section{Preliminaries}
Throughout the paper, we denote the ring of rational integers by $\mathbb{Z}$, 
the field of real numbers by $\mathbb{R}$, the field of complex numbers by $\mathbb{C}$. 
Further, we fix a positive integer $r$ and denote the partition set of length $r$ by 
\begin{equation}
\mathscr{P}:=\{\mathbf{m}=(m_{1}, \ldots, m_{r}) \in \mathbb{Z}_{\geq 0}^{r}\mid m_{1}\geq \cdots \geq m_{r}\}.
\end{equation}
For any vector $\mathbf{s}=(s_{1},\ldots,s_{r})\in \mathbb{C}^{r}$, we put 
\begin{align}
\Re{\mathbf{s}}&:=(\Re{s_{1}},\ldots,\Re{s_{r}}), \\
|\mathbf{s}|&:=s_{1}+\cdots+s_{r}, \\
\|\mathbf{s}\|&:=(|s_{1}|,\ldots,|s_{r}|).
\end{align}
Moreover, for $\mathbf{m} \in \mathscr{P}$ 
$$
\mathbf{m}!:=m_{1}!\cdots m_{r}!
$$
and we set $\delta:=(r-1,r-2,\ldots,1,0)$.
Refer to Faraut and Koranyi \cite{FK} for the details in this section.

\subsection{Analysis on symmetric cones}
Let $\Omega$ be an irreducible symmetric cone in $V$ which is a finite dimensional simple Euclidean Jordan algebra of dimension $n$ as a real vector space and rank $r$. 
The classification of irreducible symmetric cones is well-known. 
Namely, there are four families of classical irreducible symmetric cones $\Pi_{r}(\mathbb{R}), \Pi_{r}(\mathbb{C}), \Pi_{r}(\mathbb{H})$, 
the cones of all $r\times r$ positive definite matrices over $\mathbb{R}$, $\mathbb{C}$ and $\mathbb{H}$, the Lorentz cones $\Lambda_{r}$ and an exceptional cone $\Pi_{3}(\mathbb{O})$ (see \cite{FK}\,p.\,97). 
Also, let $V^{\mathbb{C}}$ be the complexification of $V$. 
For $w,z\in V^{\mathbb{C}}$, we define
\begin{align}
L(w)z&:=wz, \nonumber \\
w\Box{z}&:=L(wz)+[L(w),L(z)], \nonumber \\
P(w,z)&:=L(w)L(z)+L(z)L(w)-L(wz), \nonumber \\
P(w)&:=P(w,w)=2L(w)^{2}-L(w^{2}). \nonumber 
\end{align}
We denote the Jordan trace and determinant of the complex Jordan algebra $V^{\mathbb{C}}$ by $\tr{x}$ and by $\Delta(x)$ respectively.

Fix a Jordan frame $\{c_{1},\ldots,c_{r}\}$ that is a complete system of orthogonal primitive idempotents in $V$ and define the following subspaces:
\begin{align}
V_{j}&:=\{x \in V \mid L(c_{j})x=x\}, \nonumber \\
V_{jk}&:=\left\{x \in V \bigg| L(c_{j})x=\frac{1}{2}x \,\,\text{and}\,\, L(c_{k})x=\frac{1}{2}x\right\}. \nonumber
\end{align}
Then, $V_{j}=\mathbb{R}e_{j}$ for $j=1,\ldots,r$ are $1$-dimensional subalgebras of $V$, 
while the subspaces $V_{jk}$ for $j,k=1,\ldots,r$ with $j<k$ all have a common dimension $d=\dim_{\mathbb{R}}V_{jk}$. 
Then, $V$ has the Peirce decomposition 
\begin{equation}
V=\left(\bigoplus_{j=1}^{r}{V_{j}}\right)\oplus\left(\bigoplus_{j<k}{V_{jk}}\right), \nonumber 
\end{equation}
which is the orthogonal direct sum. 
It follows that $n=r+\frac{d}{2}r(r-1)$. 
Let $G(\Omega)$ denote the automorphism group of $\Omega$ and let $G$ be the identity component in $G(\Omega)$. 
Then, $G$ acts transitively on $\Omega$ and $\Omega \cong G/K$ where $K \in G$ is the isotropy subgroup of the unit element, $e \in V$. 
$K$ is also the identity component in $\Aut(V)$.

For any $x \in V$, there exist $k \in K$ and $\lambda_{1},\ldots, \lambda_{r} \in \mathbb{R}$ such that 
\begin{equation}
x={k}\sum_{j=1}^{r}{\lambda_{j}c_{j}},\,\,\,\,(\lambda_{1}\geq \cdots \geq \lambda_{r}). \nonumber 
\end{equation}

As in the case of $V$, we also have the following spectral decomposition for $V^{\mathbb{C}}$. 
Every $z$ in $V^{\mathbb{C}}$ can be written
$$
z=u\sum_{j=1}^{r}\lambda_{j}c_{j},
$$
with $u$ in $U$ which is the identity component of $Str(V^{\mathbb{C}})\cap{U(V^{\mathbb{C}})}$, $\lambda_{1}\geq \cdots\geq \lambda_{r}\geq 0$. 
Moreover, we define the spectral norm of $z \in V^{\mathbb{C}}$ by $|z|=\lambda_{1}$ and introduce the open unit ball $\mathcal{D} \in V^{\mathbb{C}}$ as follows. 
\begin{equation}
\mathcal{D}=\{z \in V^{\mathbb{C}} \mid |z|<1\}. \nonumber
\end{equation}

For $j=1,\ldots,r$, let $e_{j}:=c_{1}+\cdots+c_{j}$, and set 
\begin{equation}
V^{(j)}:=\{x \in V \mid L(e_{j})x=x\}. \nonumber
\end{equation}
Denote the orthogonal projection of $V$ onto the subalgebra $V^{(j)}$ by $P_{j}$, and define 
\begin{equation}
\Delta_{j}(x):=\delta_{j}(P_{j}x) \nonumber   
\end{equation}
for $x \in V$, where $\delta_{j}$ denotes the determinant with respect to $V^{(j)}$. 
In particular, $\delta_{r}=\Delta$. 
Then, $\Delta_{j}$ is a polynomial on $V$ that is homogeneous of degree $j$. 
Let $\mathbf{s}:=(s_{1},\ldots,s_{r}) \in \mathbb{C}^{r}$ and define the function $\Delta_{\mathbf{s}}$ on $V$ by 
\begin{equation}
\Delta_{\mathbf{s}}(x):=\Delta(x)^{s_{r}}\prod_{j=1}^{r-1}\Delta_{j}(x)^{s_{j}-s_{j+1}}. 
\end{equation}
That is the generalized power function on $V$. 
Furthermore, for $\mathbf{m} \in \mathscr{P}$, $\Delta_{\mathbf{m}}$ becomes a polynomial function on $V$, which is homogeneous of degree $|\mathbf{m}|$.

The gamma function $\Gamma_{\Omega}$ for the symmetric cone $\Omega$ is defined, for $\mathbf{s} \in \mathbb{C}^{r}$, with $\Re{s_{j}}>\frac{d}{2}(j-1)\,(j=1,\ldots,r)$ by 
\begin{equation}
\Gamma_{\Omega}(\mathbf{s}):=\int_{\Omega}e^{-{\rm{tr}}(x)}\Delta_{\mathbf{s}}(x)\Delta(x)^{-\frac{n}{r}}\,dx.
\end{equation}
Its evaluation gives 
\begin{equation}
\label{eq:def of gamma on cone}
\Gamma_{\Omega}(\mathbf{s})=(2\pi)^{\frac{n-r}{2}}\prod_{j=1}^{r}\Gamma\left(s_{j}-\frac{d}{2}(j-1)\right).
\end{equation}
Hence, $\Gamma_{\Omega}$ extends analytically as a meromorphic function on $\mathbb{C}^{r}$.

For $\mathbf{s} \in \mathbb{C}^{r}$ and $\mathbf{m} \in \mathscr{P}$, we define the generalized shifted factorial by 
\begin{equation}
(\mathbf{s})_{\mathbf{m}}:=\frac{\Gamma_{\Omega}(\mathbf{s}+\mathbf{m})}{\Gamma_{\Omega}(\mathbf{s})}.
\end{equation}
It follows from (\ref{eq:def of gamma on cone}) that 
\begin{equation}
(\mathbf{s})_{\mathbf{m}}=\prod_{j=1}^{r}\left(s_{j}-\frac{d}{2}(j-1)\right)_{m_{j}}.
\end{equation}
\begin{lem}
\label{thm:ineq for generalized shifted factorial}
If $\mathbf{s} \in \mathbb{C}^{r}, \mathbf{m},\mathbf{k} \in \mathscr{P}$ and $\mathbf{m}\supset \mathbf{k}$, then
\begin{equation}
\label{eq:ineq for generalized shifted factorial}
\left|\frac{(\mathbf{s})_{\mathbf{m}}}{(\mathbf{s})_{\mathbf{k}}}\right| \leq \frac{(\|\mathbf{s}\|+d(r-1))_{\mathbf{m}}}{(\|\mathbf{s}\|+d(r-1))_{\mathbf{k}}}.
\end{equation}
\end{lem}
\begin{proof}
We remark that for any $s \in \mathbb{C}, N \in \mathbb{Z}_{\geq 0}$ and $j=1,\ldots,r$, the following is satisfied.
\begin{equation}
\label{eq:ineq for generalized shifted factorial prot}
\left|s+N-\frac{d}{2}(j-1)\right|\leq |s|+N+d(r-1)-\frac{d}{2}(j-1)
=|s|+N+\frac{d}{2}(2r-j-1). \nonumber
\end{equation}
Hence, 
\begin{align}
\left|\frac{(\mathbf{s})_{\mathbf{m}}}{(\mathbf{s})_{\mathbf{k}}}\right|
&=\prod_{j=1}^{r}\left|\left(s_{j}+k_{j}-\frac{d}{2}(j-1)\right)_{m_{j}-k_{j}}\right| \nonumber \\
&\leq \prod_{j=1}^{r}\left(|s_{j}|+k_{j}+d(r-1)-\frac{d}{2}(j-1)\right)_{m_{j}-k_{j}} \nonumber \\
&=\frac{\left(\|\mathbf{s}\|+d(r-1)\right)_{\mathbf{m}}}{\left(\|\mathbf{s}\|+d(r-1)\right)_{\mathbf{k}}}. \nonumber 
\end{align}
\end{proof}
\begin{cor}
\label{thm:ineq for generalized shifted factorial 2}
If $\mathbf{s} \in \mathbb{C}^{r}, \mathbf{m} \in \mathscr{P}$, then
\begin{equation}
\label{eq:ineq for generalized shifted factorial2}
|(\mathbf{s})_{\mathbf{m}}| \leq (\|\mathbf{s}\|+d(r-1))_{\mathbf{m}} \leq \prod_{j=1}^{r}(|s_{j}|+d(r-1))_{m_{j}}.
\end{equation}
\end{cor}

The space $\mathcal{P}(V)$ of polynomials on $V$ has the following decomposition. 
\begin{equation}
\mathcal{P}(V)=\bigoplus_{\mathbf{m} \in \mathscr{P}}\mathcal{P}_{\mathbf{m}}, \nonumber 
\end{equation}
where the subspaces $\mathcal{P}_{\mathbf{m}}$ are mutually inequivalent, and finite dimensional irreducible $G$-modules. 
Further, their dimensions are denoted by $d_{\mathbf{m}}$. 
For $d_{\mathbf{m}}$, the following formula is known (see, \cite{Up}\,Lemma 2.6 or \cite{FK}\,p.\,315). 
\begin{lem}
\label{thm:Upmeier Lem 2.6}
For any $\mathbf{m} \in \mathscr{P}$,
\begin{align}
\label{eq:Upmeier Lem 2.6}
d_{\mathbf{m}}
&=\frac{c(-\rho)}{c(\rho -\mathbf{m})c(\mathbf{m}-\rho)} \\
&=\prod_{1\leq p<q\leq r}\frac{m_{p}-m_{q}+\frac{d}{2}(q-p)}{\frac{d}{2}(q-p)}
\frac{B\left(m_{p}-m_{q},\frac{d}{2}(q-p-1)+1\right)}{B\left(m_{p}-m_{q},\frac{d}{2}(q-p+1)\right)} \\
&=\prod_{j=1}^{r}\frac{\Gamma\left(\frac{d}{2}\right)}{\Gamma\left(\frac{d}{2}j\right)\Gamma\left(\frac{d}{2}(j-1)+1\right)} \nonumber \\
{} & \quad \cdot
\prod_{1\leq p<q\leq r}\left(m_{p}-m_{q}+\frac{d}{2}(q-p)\right)\frac{\Gamma\left( m_{p}-m_{q}+\frac{d}{2}(q-p+1)\right)}{\Gamma\left(m_{p}-m_{q}+\frac{d}{2}(q-p-1)+1\right)}.
\end{align}
Here, $\rho=(\rho_{1},\ldots,\rho_{r})$, $\rho_{j}:=\frac{d}{4}(2j-r-1)$, and $c$ is the Harish-Chandra function:
$$
c(\mathbf{s})=\prod_{1\leq p<q\leq r}\frac{B\left(s_{q}-s_{p}, \frac{d}{2}\right)}{B\left(\frac{d}{2}(q-p), \frac{d}{2}\right)}.
$$
In particular, for $d=2$ 
\begin{equation}
\label{eq:d and Schur}
d_{\mathbf{m}}=\prod_{1\leq p<q\leq r}\left(\frac{m_{p}-m_{q}+q-p}{q-p}\right)^{2}=s_{\mathbf{m}}(1,\ldots,1)^{2}.
\end{equation}
Here, $s_{\mathbf{m}}$ is the Schur polynomial corresponding to $\mathbf{m} \in \mathscr{P}$ defined by 
$$
s_{\mathbf{m}}(\lambda_{1},\ldots,\lambda_{r}):=\frac{\det(\lambda_{j}^{m_{k}+r-k})}{\det(\lambda_{j}^{r-k})}.
$$
\end{lem}
The following lemma is necessary to evaluate the Laplace transform of the multivariate Laguerre polynomial.
\begin{lem}[$\cite{FK}$\,Theorem\,XI.\,$2.3$]
\label{thm:int formula}
For $p \in \mathcal{P}_{\mathbf{m}}$, $\Re{\alpha}>(r-1)\frac{d}{2}$, and $y \in \Omega +iV$,
\begin{equation}
\int_{\Omega}e^{-(y|x)}p(x)\Delta(x)^{\alpha -\frac{n}{r}}\,dx=\Gamma_{\Omega}(\mathbf{m}+\alpha)\Delta(y)^{-\alpha}p(y^{-1}).
\end{equation}
Here, $\alpha$ is regarded as $(\alpha,\ldots,\alpha) \in \mathbb{C}^{r}$. 
\end{lem}
For each $\mathbf{m} \in \mathscr{P}$, the spherical polynomial of weight $|\mathbf{m}|$ on $\Omega$ is defined by 
\begin{equation}
\Phi_{\mathbf{m}}^{(d)}(x):=\int_{K}\Delta_{\mathbf{m}}(kx)\,dk.
\end{equation}
We will often omit the multiplicity $d$ and simply write $\Phi_{\mathbf{m}}$. 
The algebra of all $K$-invariant polynomials on $V$, denoted by $\mathcal{P}(V)^{K}$, decomposes as 
\begin{equation}
\mathcal{P}(V)^{K}=\bigoplus_{\mathbf{m} \in \mathscr{P}}\mathbb{C}\Phi_{\mathbf{m}}. \nonumber 
\end{equation}
By analytic continuation to the complexification $V^{\mathbb{C}}$ of $V$, 
we can extend $\tr, \Delta$ and $\Phi_{\mathbf{m}}$ to polynomial functions on $V^{\mathbb{C}}$. 
\begin{rmk}
{\rm(1)}\,Since $\Phi_{\mathbf{m}}\in\mathcal{P}_{\mathbf{m}}^{K}$, for $x=k\sum_{j=1}^{r}\lambda_{j}c_{j}$, $\Phi_{\mathbf{m}}(x)$ can be expressed by
$$
\Phi_{\mathbf{m}}(\lambda_{1},\ldots,\lambda_{r}):=\Phi_{\mathbf{m}}\left(\sum_{j=1}^{r}\lambda_{j}c_{j}\right)(=\Phi_{\mathbf{m}}(x)).
$$
$\Phi_{\mathbf{m}}(x)$ also has the following expression (see \cite{F}). 
\begin{equation}
\Phi_{\mathbf{k}}^{(d)}(\lambda_{1},\ldots,\lambda_{r})=\frac{P_{\mathbf{k}}^{(\frac{2}{d})}(\lambda_{1},\ldots,\lambda_{r})}{P_{\mathbf{k}}^{(\frac{2}{d})}(1,\ldots,1)}.
\end{equation}
Here, $P_{\mathbf{k}}^{(\frac{2}{d})}(\lambda_{1},\ldots,\lambda_{r})$ is an $r$-variable Jack polynomial (see \cite{M}, Chapter. VI.10). 
In particular, since $P_{\mathbf{k}}^{(1)}(\lambda_{1},\ldots,\lambda_{r})=s_{\mathbf{m}}(\lambda_{1},\ldots,\lambda_{r})$, $\Phi_{\mathbf{m}}^{(2)}$ becomes the Schur polynomial. 
\begin{equation}
\label{eq:spherical and schur}
\Phi_{\mathbf{m}}^{(2)}(\lambda_{1},\ldots,\lambda_{r})=\frac{s_{\mathbf{m}}(\lambda_{1},\ldots,\lambda_{r})}{s_{\mathbf{m}}(1,\ldots,1)}=\frac{\delta!}{\prod_{p<q}(m_{p}-m_{q}+q-p)}s_{\mathbf{m}}(\lambda_{1},\ldots,\lambda_{r}).
\end{equation}
\noindent
{\rm(2)}\,When $r=2$, $\Phi_{\mathbf{m}}^{(d)}$ has the following hypergeometric expression\,(see \cite{Sa}). 
\begin{align}
\Phi_{m_{1},m_{2}}^{(d)}(\lambda_{1},\lambda_{2})
&=\lambda_{1}^{m_{1}}\lambda_{2}^{m_{2}}{_{2}F_1}\left(\begin{matrix}-(m_{1}-m_{2}),\frac{d}{2}\\d \end{matrix};\frac{\lambda_{1}-\lambda_{2}}{\lambda_{1}}\right) \nonumber \\
&=\lambda_{1}^{m_{1}}\lambda_{2}^{m_{2}}\frac{\left(\frac{d}{2}\right)_{m_{1}-m_{2}}}{(d)_{m_{1}-m_{2}}}{_{2}F_1}\left(\begin{matrix}-(m_{1}-m_{2}),\frac{d}{2}\\-(m_{1}-m_{2})-\frac{d}{2}+1 \end{matrix};\frac{\lambda_{2}}{\lambda_{1}}\right). \nonumber 
\end{align}
\end{rmk}
We remark that the function $\Phi_{\mathbf{m}}(e+x)$ is a $K$-invariant polynomial of degree $|\mathbf{m}|$ and define the generalized binomial coefficients $\binom{\mathbf{m}}{\mathbf{k}}_{\frac{d}{2}}$ by using the following expansion.
\begin{equation}
\Phi_{\mathbf{m}}^{(d)}(e+x)=\sum_{|\mathbf{k}|\leq |\mathbf{m}|}\binom{\mathbf{m}}{\mathbf{k}}_{\frac{d}{2}}\Phi_{\mathbf{k}}^{(d)}(x).
\end{equation}
For $\binom{\mathbf{m}}{\mathbf{k}}_{\frac{d}{2}}$, we also often omit $\frac{d}{2}$. 
The fact that if $\mathbf{k} \not\subset \mathbf{m}$, then $\binom{\mathbf{m}}{\mathbf{k}}=0$, is well known. 
Hence, we have 
\begin{equation}
\label{eq:the definition of the generalized binomial coefficients}
\Phi_{\mathbf{m}}(e+x)=\sum_{\mathbf{k} \subset \mathbf{m}}\binom{\mathbf{m}}{\mathbf{k}}\Phi_{\mathbf{k}}(x).
\end{equation}

%
Moreover, for the spherical polynomials, we refer two Lemmas in \cite{FK}. 
\begin{lem}[$\cite{FK}$\,Theorem\,XII.\,$1.1$\,(i)]
\label{thm:FK,Thm12.1.1}
For $z=u\sum_{j=1}^{r}\lambda_{j}c_{j}$ with $u \in U$, $\lambda_{1}\geq \cdots\geq \lambda_{r}\geq 0$ and $\mathbf{m} \in \mathscr{P}$, we have
\begin{equation}
\label{eq:FK,Thm12.1.1}
|\Phi_{\mathbf{m}}(z)|\leq \lambda_{1}^{m_{1}}\cdots \lambda_{r}^{m_{r}}\leq \lambda_{1}^{|\mathbf{m}|}=\Phi_{\mathbf{m}}(\lambda_{1}).
\end{equation}
\end{lem}
\begin{lem}[$\cite{FK}$\,Chapter\,XV.\,Exercise\,$3$\,(a)]
\label{thm:Cauchy kernel of spherical poly}
For any $\alpha \in \mathbb{C},z \in \overline{\mathcal{D}},w \in \mathcal{D}$, we have
\begin{equation}
\label{eq:Cauchy kernel of spherical poly}
\sum_{\mathbf{m} \in \mathscr{P}}d_{\mathbf{m}}\frac{(\alpha)_{\mathbf{m}}}{\left(\frac{n}{r}\right)_{\mathbf{m}}}\Phi_{\mathbf{m}}(z)\Phi_{\mathbf{m}}(w)
=\Delta(w)^{-\alpha}\int_{K}\Delta(kw^{-1}-z)^{-\alpha}\,dk.
\end{equation}
\end{lem}

The spherical function, $\varphi_{\mathbf{s}}$, on $\Omega$ for $\mathbf{s} \in \mathbb{C}^{r}$ is defined by 
\begin{equation}
\varphi_{\mathbf{s}}(x):=\int_{K}\Delta_{\mathbf{s}+\rho}(kx)\,dk.  
\end{equation}
We remark that for $x \in \Omega$
\begin{equation}
\varphi_{\mathbf{s}}(x^{-1})=\varphi_{-\mathbf{s}}(x)
\end{equation}
and  for $x \in \Omega, \mathbf{m} \in \mathscr{P}$
\begin{equation}
\Phi_{\mathbf{m}}(x)=\varphi_{\mathbf{m}-\rho}(x).  
\end{equation}

Let $\mathbb{D}(\Omega)$ be the algebra of $G$-invariant differential operators on $\Omega$, 
$\mathcal{P}(V)^{K}$ be the space of $K$-invariant polynomials on $V$, 
and $\mathcal{P}(V\times V)^{G}$ be the space of polynomials on $V\times V$, which are invariant in the sense that
$$
p(gx,\xi)=p(x,g^{*}\xi),\,\,\,\,(g \in G).
$$
Here, we write $g^{*}$ for the adjoint of an element $g$ (i.e., $(gx|y)=(x|g^{*}y)$ for all $x,y \in V$). 
The spherical function $\varphi_{\mathbf{s}}$ is an eigenfunction of every $D \in \mathbb{D}(\Omega)$. 
Thus, we denote its eigenvalues by $\gamma(D)(\mathbf{s})$, that is, $D\varphi_{\mathbf{s}}=\gamma(D)(\mathbf{s})\varphi_{\mathbf{s}}$. 

The symbol $\sigma_{D}$ of a partial differential operator $D$ which acts on the variable $x \in V$ is defined by 
\begin{equation}
De^{(x|\xi)}=\sigma_{D}(x,\xi)e^{(x|\xi)}\,\,\,\,(x,\xi \in V). \nonumber
\end{equation}
A differential operator $D$ on $\Omega$ is invariant under $G$ if and only if its symbol $\sigma_{D}$ belongs to $\mathcal{P}(V\times V)^{G}$. 
In addition, the map $D \mapsto \sigma_{D}$ establishes a linear isomorphism 
from $\mathbb{D}(\Omega)$ onto $\mathcal{P}(V\times V)^{G}$. 
Moreover, the map $D \mapsto \sigma_{D}(e,u)$ is a vector space isomorphism from $\mathbb{D}(\Omega)$ onto $\mathcal{P}(V)^{K}$. 
In particular, for $\mathbf{k} \in \mathscr{P}, \mathbf{s} \in \mathbb{C}^{r}$, we put
\begin{equation}
\gamma_{\mathbf{k}}(\mathbf{s}):=\gamma(\Phi_{\mathbf{k}}(\partial_{x}))(\mathbf{s})=\Phi_{\mathbf{k}}(\partial_{x})\varphi_{\mathbf{s}}(x)|_{x=e}.
\end{equation}
Here, $\Phi_{\mathbf{k}}(\partial_{x})$ is the unique $G$-invariant differential operator satisfying 
$$
\sigma_{\Phi_{\mathbf{k}}(\partial_{x})}(e,\xi)=\Phi_{\mathbf{k}}(\xi) \in \mathcal{P}(V)^{K},
\,\,\,\,\text{i.e.,}\,\,\Phi_{\mathbf{k}}(\partial_{x})e^{(x|\xi)}|_{x=e}=\Phi_{\mathbf{k}}(\xi)e^{\tr{\xi}}.
$$ 
We remark that $\Phi_{k}(\partial_{x})=\partial_{x}^{k}$ and $\gamma_{k}(s)=s(s-1)\cdots(s-k+1)$ in the $r=1$ case, and for any $\alpha \in \mathbb{C}$, $\mathbf{k} \in \mathscr{P}$, we have
\begin{equation}
\label{eq:gamma special case}
\gamma_{\mathbf{k}}(\alpha -\rho)=(-1)^{|\mathbf{k}|}(-\alpha)_{\mathbf{k}}.
\end{equation}   
The function $\gamma_{D}$ is an $r$ variable symmetric polynomial   
and map $D \mapsto \gamma_{D}$ is an algebra isomorphism from $\mathbb{D}(\Omega)$ onto the algebra $\mathcal{P}(\mathbb{R}^{r})^{\mathfrak{S}_{r}}$, 
which is a special case of the Harish-Chandra isomorphism. 
If a $K$-invariant function $\psi$ is analytic in the neighborhood of $e$, it admits a spherical Taylor expansion near $e$:
\begin{equation}
\psi(e+x)=\sum_{\mathbf{k} \in \mathscr{P}}d_{\mathbf{k}}\frac{1}{\left(\frac{n}{r}\right)_{\mathbf{k}}}\{\Phi_{\mathbf{k}}(\partial_{x})\psi(x)|_{x=e}\}\Phi_{\mathbf{k}}(x). \nonumber
\end{equation}
By the definition of $\gamma_{\mathbf{k}}$, we have
\begin{equation}
\varphi_{\mathbf{s}}(e+x)=\sum_{\mathbf{k} \in \mathscr{P}}d_{\mathbf{k}}\frac{1}{\left(\frac{n}{r}\right)_{\mathbf{k}}}\gamma_{\mathbf{k}}(\mathbf{s})\Phi_{\mathbf{k}}(x). \nonumber
\end{equation}
Since $\Phi_{\mathbf{m}}=\varphi_{\mathbf{m}-\rho}$, 
$$
\binom{\mathbf{m}}{\mathbf{k}}=d_{\mathbf{k}}\frac{1}{\left(\frac{n}{r}\right)_{\mathbf{k}}}\gamma_{\mathbf{k}}(\mathbf{m}-\rho).
$$

For a complex number $\alpha$, we define the following differential operator on $\Omega$:
\begin{equation}
D_{\alpha}=\Delta(x)^{1+\alpha}\Delta(\partial_{x})\Delta(x)^{-\alpha}. \nonumber
\end{equation}
For this operator, we have 
\begin{equation}
\gamma(D_{\alpha})(\mathbf{s})=\prod_{j=1}^{r}\left(s_{j}-\alpha+\frac{d}{4}(r-1)\right).
\end{equation}
The operators $D_{j\frac{d}{2}},\,j=0,\ldots,r-1$ generate the algebra $\mathbb{D}(\Omega)$. 
\begin{lem}
\label{thm:estimate of shifted Jack1}
For all $\mathbf{k} \in \mathscr{P}$, there exist some constant $C>0$ and integer $N$ such that for any $\mathbf{s} \in \mathbb{C}^{r}$
\begin{align}
\label{eq:estimate of shifted Jack1}
|\gamma_{\mathbf{k}}(\mathbf{s})|&\leq C\prod_{l=1}^{r}\left(|s_{l}|+\frac{d}{4}(r-1)\right)^{N}. 
\end{align}
\end{lem}
\begin{proof}
Since the algebra $\mathbb{D}(\Omega)$ is generated by $D_{j\frac{d}{2}},\,j=0,\ldots,r-1$, 
for $\Phi_{\mathbf{k}}(\partial_{x}) \in \mathbb{D}(\Omega)$,  
$$
\Phi_{\mathbf{k}}(\partial_{x})=\sum_{l_{0},\ldots,l_{r-1};{\text{finite}}}
a_{l_{0},\ldots,l_{r-1}}D_{0\frac{d}{2}}^{l_{0}} \cdots D_{(r-1)\frac{d}{2}}^{l_{r-1}}.
$$
Here, we remark that for $j=0,\ldots,r-1$
\begin{align}
|\gamma(D_{\frac{d}{2}(j-1)})(\mathbf{s})|
=\left|\prod_{l=1}^{r}\left(s_{l}+\frac{d}{4}(r-1)-\frac{d}{2}(j-1)\right)\right| 
\leq \prod_{l=1}^{r}\left(|s_{l}|+\frac{d}{4}(r-1)\right). \nonumber
\end{align}
Therefore, 
\begin{align}
|\gamma_{\mathbf{k}}(\mathbf{s})|
\leq \sum_{l_{0},\ldots,l_{r-1};{\text{finite}}}
|a_{l_{0},\ldots,l_{r-1}}|\gamma(D_{0\frac{d}{2}})(\mathbf{s})^{l_{0}} \cdots \gamma(D_{(r-1)\frac{d}{2}})(\mathbf{s})^{l_{r-1}}
\leq C\prod_{l=1}^{r}\left(|s_{l}|+\frac{d}{4}(r-1)\right)^{N}. \nonumber
\end{align}
\end{proof}
\begin{lem}
\label{thm:positivity of shifted Jack}
For all $\mathbf{m},\mathbf{k} \in \mathscr{P}$, we have
\begin{equation}
\gamma_{\mathbf{k}}(\mathbf{m}-\rho)\geq 0.
\end{equation}
\end{lem}
\begin{proof}
Since $\gamma_{\mathbf{k}}(\mathbf{m}-\rho)=\frac{1}{d_{\mathbf{k}}}\left(\frac{n}{r}\right)_{\mathbf{k}}\binom{\mathbf{m}}{\mathbf{k}}$ 
and $d_{\mathbf{k}}, \left(\frac{n}{r}\right)_{\mathbf{k}}>0$, it suffices to show $\binom{\mathbf{m}}{\mathbf{k}}\geq 0$ for all $\mathbf{m},\mathbf{k} \in \mathscr{P}$. 
From \cite{OO}, the generalized binomial coefficients are written as 
$$
\binom{\mathbf{m}}{\mathbf{k}}_{\frac{d}{2}}=\frac{P_{\mathbf{k}}^{\ast}\left(\mathbf{m};\frac{d}{2}\right)}{H_{\left(\frac{d}{2}\right)}(\mathbf{k})},
$$
where $P_{\mathbf{k}}^{\ast}\left(\mathbf{m};\frac{d}{2}\right)$ is the shifted Jack polynomial (see also \cite{Sahi}, \cite{KS} and \cite{OO})
and $H_{(\frac{d}{2})}(\mathbf{k})>0$ is a deformation of the hook length. 
Moreover, by using (5.2) in \cite{OO}
$$
P_{\mathbf{k}}^{\ast}\left(\mathbf{m};\frac{d}{2}\right)
=\frac{\frac{d}{2}\text{-}\dim{\mathbf{m}/\mathbf{k}}}{\frac{d}{2}\text{-}\dim{\mathbf{m}}}|\mathbf{m}|(|\mathbf{m}|-1)\cdots(|\mathbf{m}|-|\mathbf{k}|+1).
$$
Further, the positivity of the generalized dimensions of the skew Young diagram, $\frac{d}{2}\text{-}\dim{\mathbf{m}/\mathbf{k}}$, follows from (5.1) of \cite{OO} and Chapter V\hspace{-.1em}I.\,6  of \cite{M}. 
Therefore, we obtain the positivity of the shifted Jack polynomial and the conclusion. 
\end{proof}
\begin{thm}
\label{thm:spherical Taylor expan lem}
{\rm{(1)}}\,
For $w \in \mathcal{D}, \mathbf{k} \in \mathscr{P}, \alpha \in \mathbb{C}$, we have
\begin{equation}
\label{eq:basic expansion 1}
(\alpha)_{\mathbf{k}}\Delta(e-w)^{-\alpha}\Phi_{\mathbf{k}}(w(e-w)^{-1})
=\sum_{\mathbf{x} \in \mathscr{P}}d_{\mathbf{x}}\frac{(\alpha)_{\mathbf{x}}}{\left(\frac{n}{r}\right)_{\mathbf{x}}}\gamma_{\mathbf{k}}(\mathbf{x}-\rho)\Phi_{\mathbf{x}}(w). 
\end{equation}
Here, we choose the branch of $\Delta (e-w)^{-\alpha}$ which takes the value $1$ at $w=0$.

\noindent
{\rm(2)}\,For $w \in V^{\mathbb{C}}, \mathbf{k} \in \mathscr{P}$, the $K$-invariant analytic function $e^{\tr{w}}\Phi_{\mathbf{k}}(w)$ has the following expansion 
\begin{equation}
\label{eq:basic expansion 2}
e^{\tr{w}}\Phi_{\mathbf{k}}(w)=\sum_{\mathbf{x} \in \mathscr{P}}d_{\mathbf{x}}\frac{1}{\left(\frac{n}{r}\right)_{\mathbf{x}}}\gamma_{\mathbf{k}}(\mathbf{x}-\rho)\Phi_{\mathbf{x}}(w).
\end{equation}
\end{thm}
\begin{proof}
{\rm{(1)}}\,We take $w=u\sum_{j=1}^{r}\lambda_{j}c_{j} \in \mathcal{D}$ with $u \in U$ and $1>\lambda_{1}\geq \ldots\geq \lambda_{r}\geq 0$. 
By Lemmas\,\ref{thm:FK,Thm12.1.1} and \ref{thm:estimate of shifted Jack1}, there exist some $C>0$ and $N \in \mathbb{Z}_{\geq 0}$ such that
\begin{align}
\sum_{\mathbf{x} \in \mathscr{P}}\left|d_{\mathbf{x}}\frac{(\alpha)_{\mathbf{x}}}{\left(\frac{n}{r}\right)_{\mathbf{x}}}\gamma_{\mathbf{k}}(\mathbf{x}-\rho)\Phi_{\mathbf{x}}(w)\right|
&\leq \sum_{\mathbf{x} \in \mathscr{P}}d_{\mathbf{x}}\frac{|(\alpha)_{\mathbf{x}}|}{\left(\frac{n}{r}\right)_{\mathbf{x}}}|\gamma_{\mathbf{k}}(\mathbf{x}-\rho)||\Phi_{\mathbf{x}}(w)| \nonumber \\
&\leq C\prod_{l=1}^{r}\sum_{x_{l}\geq 0}\frac{(|\alpha|+d(r-1))_{x_{l}}}{x_{l}!}\left(x_{l}+\frac{d}{2}(r-1)\right)^{N}\lambda_{l}^{x_{l}}< \infty. \nonumber
\end{align}
Therefore, the right hand side of (\ref{eq:basic expansion 1}) converges absolutely. 
By analytic continuation, it is sufficient to show the assertion when $\Re{\alpha}>\frac{d}{2}(r-1)$ and $w \in \Omega \cap (e-\Omega) \subset \mathcal{D}$.
\begin{align}
\Phi_{\mathbf{k}}(\partial_{z})\sum_{\mathbf{x} \in \mathscr{P}}d_{\mathbf{x}}\frac{(\alpha)_{\mathbf{x}}}{\left(\frac{n}{r}\right)_{\mathbf{x}}}\Phi_{\mathbf{x}}(z)\Phi_{\mathbf{x}}(w)\bigg|_{z=e}
&=\sum_{\mathbf{x} \in \mathscr{P}}d_{\mathbf{x}}\frac{(\alpha)_{\mathbf{x}}}{\left(\frac{n}{r}\right)_{\mathbf{x}}}\Phi_{\mathbf{k}}(\partial_{z})\Phi_{\mathbf{x}}(z)|_{z=e}\Phi_{\mathbf{x}}(w) \nonumber \\
&=\sum_{\mathbf{x} \in \mathscr{P}}d_{\mathbf{x}}\frac{(\alpha)_{\mathbf{x}}}{\left(\frac{n}{r}\right)_{\mathbf{x}}}\gamma_{\mathbf{k}}(\mathbf{x}-\rho)\Phi_{\mathbf{x}}(w). \nonumber
\end{align}
On the other hand, 
\begin{align}
\Phi_{\mathbf{k}}(\partial_{z})\sum_{\mathbf{x} \in \mathscr{P}}d_{\mathbf{x}}\frac{(\alpha)_{\mathbf{x}}}{\left(\frac{n}{r}\right)_{\mathbf{x}}}\Phi_{\mathbf{x}}(z)\Phi_{\mathbf{x}}(w)\bigg|_{z=e}
&=\Phi_{\mathbf{k}}(\partial_{z})\Delta(w)^{-\alpha}\int_{K}\Delta(kw^{-1}-z)^{-\alpha}\,dk\bigg|_{z=e} \nonumber \\
&=\Delta(w)^{-\alpha}\int_{K}\Phi_{\mathbf{k}}(\partial_{z})\Delta(kw^{-1}-z)^{-\alpha}\big|_{z=e}\,dk. \nonumber 
\end{align}
Here, from $kw^{-1}-z \in T_{\Omega}$ for all $k \in K$ and Lemma\,\ref{thm:int formula}, 
\begin{align}
\Phi_{\mathbf{k}}(\partial_{z})\Delta(kw^{-1}-z)^{-\alpha}\big|_{z=e}
&=\Phi_{\mathbf{k}}(\partial_{z})\frac{1}{\Gamma_{\Omega}(\alpha)}\int_{\Omega}e^{-(x|kw^{-1}-z)}\Delta(x)^{\alpha}\Delta(x)^{-\frac{n}{r}}\,dx\bigg|_{z=e} \nonumber \\
&=\frac{1}{\Gamma_{\Omega}(\alpha)}\int_{\Omega}\Phi_{\mathbf{k}}(\partial_{z})e^{(x|z)}|_{z=e}e^{-(x|kw^{-1})}\Delta(x)^{\alpha}\Delta(x)^{-\frac{n}{r}}\,dx \nonumber \\
&=\frac{1}{\Gamma_{\Omega}(\alpha)}\int_{\Omega}\Phi_{\mathbf{k}}(x)e^{-(kx|(w^{-1}-e))}\Delta(x)^{\alpha}\Delta(x)^{-\frac{n}{r}}\,dx \nonumber \\
&=(\alpha)_{\mathbf{k}}\Delta(w^{-1}-e)^{-\alpha}\Phi_{\mathbf{k}}((w^{-1}-e)^{-1}). \nonumber
\end{align}
Therefore, 
\begin{align}
\Phi_{\mathbf{k}}(\partial_{z})\sum_{\mathbf{x} \in \mathscr{P}}d_{\mathbf{x}}\frac{(\alpha)_{\mathbf{x}}}{\left(\frac{n}{r}\right)_{\mathbf{x}}}\Phi_{\mathbf{x}}(z)\Phi_{\mathbf{x}}(w)\bigg|_{z=e}
&=\Delta(w)^{-\alpha}\int_{K}(\alpha)_{\mathbf{k}}\Delta(w^{-1}-e)^{-\alpha}\Phi_{\mathbf{k}}((w^{-1}-e)^{-1})\,dk \nonumber \\
&=(\alpha)_{\mathbf{k}}\Delta(e-w)^{-\alpha}\Phi_{\mathbf{k}}(w(e-w)^{-1}). \nonumber
\end{align}
{\rm{(2)}}\,Since 
the right hand side of (\ref{eq:basic expansion 2}) converges absolutely due to a similar argument of {\rm{(1)}}, we have
\begin{align}
e^{\tr{w}}\Phi_{\mathbf{k}}(w)
&=\lim_{\alpha \to \infty}(\alpha)_{\mathbf{k}}\Delta\left(e-\frac{w}{\alpha}\right)^{-\alpha}\Phi_{\mathbf{k}}\left(\frac{w}{\alpha}\left(e-\frac{w}{\alpha}\right)^{-1}\right) \nonumber \\
&=\sum_{\mathbf{x} \in \mathscr{P}}d_{\mathbf{x}}\frac{1}{\left(\frac{n}{r}\right)_{\mathbf{x}}}\gamma_{\mathbf{k}}(\mathbf{x}-\rho)\lim_{\alpha \to \infty}(\alpha)_{\mathbf{x}}\Phi_{\mathbf{x}}\left(\frac{w}{\alpha}\right) \nonumber \\
&=\sum_{\mathbf{x} \in \mathscr{P}}d_{\mathbf{x}}\frac{1}{\left(\frac{n}{r}\right)_{\mathbf{x}}}\gamma_{\mathbf{k}}(\mathbf{x}-\rho)\Phi_{\mathbf{x}}(w). \nonumber
\end{align}
\end{proof}
Next we consider the gradient for a $\mathbb{C}$-valued function $f$ on a simple Euclidean Jordan algebra $V$. 
In this part we refer to \cite{Di}. 
For a scalar or vector valued differentiable function $f$ we define the gradient $\nabla f(x)$ by 
$$
(\nabla f(x)|u)=D_{u}f(x)=\frac{d}{dt}f(x+tu)\bigg|_{t=0}.
$$
For a $\mathbb{C}$-valued function $f=f_{1}+if_{2}$, we define $\nabla f=\nabla f_{1}+i\nabla f_{2}$. 
For $z=x+iy \in V^{\mathbb{C}}$, we define $D_{z}=D_{x}+iD_{y}$. 
Moreover, if $\{e_{1},\ldots,e_{n}\}$ is an orthonormal basis of $V$ and $x=\sum_{j=1}^{n}x_{j}e_{j} \in V^{\mathbb{C}}$, then 
$$
\nabla f(x)=\sum_{j=1}^{n}\frac{{\partial}f(x)}{{\partial}x_{j}}e_{j}.
$$
We remark that this expression is independent of the choice of an orthonormal basis of $V$. 

For a $V$-valued function $f:V\rightarrow V$ expressed by $f(x)=\sum_{j=1}^{r}f_{j}(x)e_{j}$, we define $\nabla f$ by
$$
\nabla f(x)=\sum_{j,l=1}^{n}\frac{{\partial}f_{j}(x)}{{\partial}x_{l}}e_{j}e_{l}.
$$
That is also well defined. 
Let us present some derivation formulas. 
\begin{lem}
\label{thm:derivation Lemma}
{\rm{(1)}}\,The product rule of differentiation: 
For $V$-valued function $f,h$, we have 
\begin{equation}
\tr{(\nabla(f(x)h(x)))}=\tr{(\nabla{f(x)})}h(x)+f(x)\tr{(\nabla{h(x)})}.
\end{equation}
For $\mathbb{C}$-valued functions $f,h$, 
\begin{equation}
\nabla(f(x)h(x))=(\nabla{f(x)})h(x)+f(x)(\nabla{h(x)}). 
\end{equation}
{\rm{(2)}}
\begin{equation}
\nabla{x}=\frac{n}{r}e.
\end{equation}
{\rm{(3)}}\,For any invertible element $x \in V^{\mathbb{C}}$, 
\begin{equation}
\tr{(x\nabla)}x^{-1}:=\tr{(x(\nabla{x^{-1}}))}=-\frac{n}{r}\tr{x^{-1}}.
\end{equation}
{\rm{(4)}}\,For $\beta \in \mathbb{C}$ and an invertible element $x \in V^{\mathbb{C}}$, 
\begin{equation}
\nabla (\Delta(x)^{\beta})=\beta \Delta(x)^{\beta}x^{-1}.
\end{equation}
\end{lem}
\noindent
{\rm{(1)}},\,{\rm{(2)}}, and {\rm{(4)}} are well known (see \cite{FK}, \cite{Di}, and \cite{FW1}). 
{\rm{(3)}} follows from {\rm{(1)}}, {\rm{(2)}}, and $\nabla(x x^{-1})=\nabla(e)=0$.
 

The following recurrence formulas for the spherical functions, some of which involve the gradient, are also well known (see \cite{Di} and \cite{FW1}). 
\begin{lem}
\label{thm:Pieri}
Let $\mathbf{s} \in \mathbb{C}^{r}$ and $x \in V^{\mathbb{C}}$. 
Put 
\begin{equation}
a_{j}(\mathbf{s}):=\frac{c(\mathbf{s})}{c(\mathbf{s}+\epsilon_{j})}=\prod_{k\not=j}\frac{s_{j}-s_{k}+\frac{d}{2}}{s_{j}-s_{k}}, 
\end{equation}
where $\epsilon_{j}:=(0,\ldots,0,\stackrel{j}{\stackrel{\vee}{1}},0,\ldots,0)$. 
Then, 
\begin{align}
\label{eq:Pieri 1}
(\tr{x})\varphi_{\mathbf{s}}(x)&=\sum_{j=1}^{r}a_{j}(\mathbf{s})\varphi_{\mathbf{s}+\epsilon_{j}}(x), \\
\label{eq:Pieri 2}
(\tr{\nabla})\varphi_{\mathbf{s}}(x)&=\sum_{j=1}^{r}\left(s_{j}+\frac{d}{4}(r-1)\right)a_{j}(-\mathbf{s})\varphi_{\mathbf{s}-\epsilon_{j}}(x), \\
\label{eq:Pieri 3}
(\tr{(x^{2}\nabla}))\varphi_{\mathbf{s}}(x)&=\sum_{j=1}^{r}\left(s_{j}-\frac{d}{4}(r-1)\right)a_{j}(\mathbf{s})\varphi_{\mathbf{s}+\epsilon_{j}}(x).
\end{align}
\end{lem}

\subsection{Multivariate Laguerre polynomials and their unitary picture}
In this subsection, we promote a unitary picture associated with the multivariate Laguerre polynomials 
and provide some fundamental lemmas based on \cite{DOZ}, \cite{FK} and \cite{FW1}.

First, we recall some function spaces and their complete orthogonal basis as in the case of one variable. 
Let $\alpha >2\frac{n}{r}-1$, $\mathbf{m} \in \mathscr{P}$.
\\
{\bf{(1)}}\,\,$\psi_{\mathbf{m}}^{(\alpha)}$\,;\,Multivariate Laguerre polynomials (up to an exponential factor)
\begin{align}
L^{2}_{\alpha}(\Omega)^{K}&:=\{\psi:\Omega \longrightarrow \mathbb{C} \mid \psi \text{ is $K$-invariant and } \|\psi\|_{\alpha,\Omega}^{2}<\infty\}, \nonumber \\
\|\psi\|_{\alpha,\Omega}^{2}&:=\frac{2^{r\alpha}}{\Gamma_{\Omega}(\alpha)}\int_{\Omega}|\psi(u)|^{2}\Delta(u)^{\alpha -\frac{n}{r}}\,du, \nonumber \\
\psi_{\mathbf{m}}^{(\alpha)}(u)&:=e^{-\tr{u}}L_{\mathbf{m}}^{(\alpha-\frac{n}{r})}(2u). \nonumber 
\end{align}
Here, $L_{\mathbf{m}}^{\left(\alpha -\frac{n}{r}\right)}(u)$ is the multivariate Laguerre polynomial defined by 
\begin{align}
L_{\mathbf{m}}^{\left(\alpha -\frac{n}{r}\right)}(u)
&:=d_{\mathbf{m}}\frac{(\alpha)_{\mathbf{m}}}{\left(\frac{n}{r}\right)_{\mathbf{m}}}
\sum_{\mathbf{k}\subset \mathbf{m}}(-1)^{|\mathbf{k}|}\binom{\mathbf{m}}{\mathbf{k}}\frac{1}{(\alpha)_{\mathbf{k}}}\Phi_{\mathbf{k}}(u) \nonumber \\
&=d_{\mathbf{m}}\frac{(\alpha)_{\mathbf{m}}}{\left(\frac{n}{r}\right)_{\mathbf{m}}}
\sum_{\mathbf{k}\subset \mathbf{m}}(-1)^{|\mathbf{k}|}d_{\mathbf{k}}\frac{\gamma_{\mathbf{k}}(\mathbf{m}-\rho)}{\left(\frac{n}{r}\right)_{\mathbf{k}}(\alpha)_{\mathbf{k}}}\Phi_{\mathbf{k}}(u). \nonumber
\end{align}

\noindent
{\bf{(2)}}\,\,$F_{\mathbf{m}}^{(\alpha)}$\,;\,Cayley transform of the spherical polynomials 
\begin{align}
\mathcal{H}^{2}_{\alpha}(T_{\Omega})^{K}&:=\{F:T_{\Omega} \longrightarrow \mathbb{C} \mid F \text{ is $K$-invariant and analytic in $T_{\Omega}$, and }\|F\|_{\alpha,T_{\Omega}}^{2}<\infty\}, \nonumber \\
\|F\|_{\alpha,T_{\Omega}}^{2}&:=\frac{1}{(4\pi)^{n}}\frac{\Gamma_{\Omega}(\alpha)}{\Gamma_{\Omega}\left(\alpha-\frac{n}{r}\right)}\int_{T_{\Omega}}|F(z)|^{2}\Delta(x)^{\alpha -\frac{2n}{r}}\,m(dz), \nonumber \\
F_{\mathbf{m}}^{(\alpha)}(z)&:=d_{\mathbf{m}}\frac{(\alpha)_{\mathbf{m}}}{\left(\frac{n}{r}\right)_{\mathbf{m}}}
\Delta\left(\frac{e+z}{2}\right)^{-\alpha}\Phi_{\mathbf{m}}((z-e)(z+e)^{-1}). \nonumber
\end{align}
{\bf{(3)}}\,\,$f_{\mathbf{m}}^{(\alpha)}$\,;\,spherical polynomials
\begin{align}
\mathcal{H}^{2}_{\alpha}(\mathcal{D})^{K}&:=\{f:\mathcal{D} \longrightarrow \mathbb{C} \mid f \text{ is $K$-invariant and analytic in $\mathcal{D}$, and }\|f\|_{\alpha,\mathcal{D}}^{2}<\infty\}, \nonumber \\
\|f\|_{\alpha,\mathcal{D}}^{2}&:=\frac{1}{{\pi}^{n}}\frac{\Gamma_{\Omega}(\alpha)}{\Gamma_{\Omega}\left(\alpha-\frac{n}{r}\right)}
\int_{\mathcal{D}}|f(w)|^{2}h(w)^{\alpha -\frac{2n}{r}}\,m(dw), \nonumber \\
h(w)&:=\Det(I_{V^{\mathbb{C}}}-2w\Box{\overline{w}}+P(w)P(\overline{w}))^{\frac{r}{2n}}, \nonumber \\
f_{\mathbf{m}}^{(\alpha)}(w)&:=d_{\mathbf{m}}\frac{(\alpha)_{\mathbf{m}}}{\left(\frac{n}{r}\right)_{\mathbf{m}}}\Phi_{\mathbf{m}}(u).\nonumber 
\end{align}
Here, $\Det$ stands for the usual determinant of a complex linear operator on $V^{\mathbb{C}}$.

We remark that
$$
\|\psi_{\mathbf{m}}^{(\alpha)}\|_{\alpha,\Omega}^{2}=\|F_{\mathbf{m}}^{(\alpha)}\|_{\alpha,T_{\Omega}}^{2}
=\|f_{\mathbf{m}}^{(\alpha)}\|_{\alpha,\mathcal{D}}^{2}=d_{\mathbf{m}}\frac{(\alpha)_{\mathbf{m}}}{\left(\frac{n}{r}\right)_{\mathbf{m}}}
$$
and the orthogonality relations of $\psi_{\mathbf{m}}^{(\alpha)}$ also hold for $\alpha >\frac{n}{r}-1$.

Next, similar to the one variable case, we will consider some unitary isomorphisms. 
\noindent
\underline{Modified Laplace transform}
\begin{align}
\mathcal{L}_{\alpha}:L^{2}_{\alpha}(\Omega)^{K} \xrightarrow{\simeq}  \mathcal{H}^{2}_{\alpha}(T_{\Omega})^{K},
\,\,\,(\mathcal{L}_{\alpha}\psi)(z):=\frac{2^{r\alpha}}{\Gamma_{\Omega}(\alpha)}\int_{\Omega}e^{-(z|u)}\Delta(u)^{\alpha-\frac{n}{r}}\psi(u)\,du.\nonumber
\end{align}
\underline{Modified Cayley transform}
\begin{align}
C_{\alpha}^{-1}:\mathcal{H}^{2}_{\alpha}(T_{\Omega})^{K} \xrightarrow{\simeq} \mathcal{H}^{2}_{\alpha}(\mathcal{D})^{K},\,\,\,(C_{\alpha}^{-1}F)(w):=\Delta(e-w)^{-\alpha}F\left((e+w)(e-w)^{-1}\right). \nonumber
\end{align}

To summarize the above, we obtain the following picture.
\begin{align}
\label{eq:unitary picture2}
\begin{array}{cccccc}
L^{2}_{\alpha}(\Omega)^{K} & \xrightarrow[\mathcal{L}_{\alpha}]{\simeq} & \mathcal{H}^{2}_{\alpha}(T_{\Omega})^{K} & \xrightarrow[C_{\alpha}^{-1}]{\simeq} & \mathcal{H}^{2}_{\alpha}(\mathcal{D})^{K}. & \\
\rotatebox{90}{$\in$}  & & \rotatebox{90}{$\in$} & & \rotatebox{90}{$\in$} & \\
\psi_{m}^{(\alpha)}  & \longmapsto & F_{m}^{(\alpha)}& \longmapsto & f_{m}^{(\alpha)} &  \\
{\bf{(1)}} &  & {\bf{(2)}} &  & {\bf{(3)}} & 
\end{array}
\end{align}
\begin{lem}
\label{thm:generating fnc of Laguerre and MP}
For any $\alpha \in \mathbb{C},u \in \Omega$ and $z \in \mathcal{D}$, we have
\begin{equation}
\label{eq:generating fnc of Laguerre}
\sum_{\mathbf{m} \in \mathscr{P}}L_{\mathbf{m}}^{\left(\alpha -\frac{n}{r}\right)}(u)\Phi_{\mathbf{m}}(z)=\Delta(e-z)^{-\alpha}\int_{K}e^{-(ku|z(e-z)^{-1})}\,dk.
\end{equation}
\end{lem}

\begin{proof}
By referring to \cite{Dib} (see Proposition\,2.\,8), (\ref{eq:generating fnc of Laguerre}) holds for $\alpha >\frac{n}{r}-1=d(r-1)$. 
Moreover, the right hand side of (\ref{eq:generating fnc of Laguerre}) is well defined for any $\alpha \in \mathbb{C}$.  
Hence, by analytic continuation, it is sufficient to show the absolute convergence of the left hand side under the assumption. 
By Lemmas\,\ref{thm:ineq for generalized shifted factorial}, \ref{thm:FK,Thm12.1.1} and \ref{thm:Cauchy kernel of spherical poly}, 
\begin{align}
\sum_{\mathbf{m} \in \mathscr{P}}|L_{\mathbf{m}}^{\left(\alpha -\frac{n}{r}\right)}(u)\Phi_{\mathbf{m}}(z)|
&\leq \sum_{\mathbf{m} \in \mathscr{P}}\sum_{\mathbf{k}\subset \mathbf{m}}
\left|d_{\mathbf{m}}\frac{(\alpha)_{\mathbf{m}}}{\left(\frac{n}{r}\right)_{\mathbf{m}}}
\binom{\mathbf{m}}{\mathbf{k}}\frac{(-1)^{\mathbf{k}}}{(\alpha)_{\mathbf{k}}}\Phi_{\mathbf{k}}(u)\right|\Phi_{\mathbf{m}}(a_{1}) \nonumber \\
&\leq \sum_{\mathbf{k} \in \mathscr{P}}d_{\mathbf{k}}\frac{1}{\left(\frac{n}{r}\right)_{\mathbf{k}}}\frac{1}{(|\alpha|+d(r-1))_{\mathbf{k}}}\Phi_{\mathbf{k}}(u) \nonumber \\
{} & \quad \sum_{\mathbf{m} \in \mathscr{P}}d_{\mathbf{m}}\frac{(|\alpha|+d(r-1))_{\mathbf{m}}}{\left(\frac{n}{r}\right)_{\mathbf{m}}}\gamma_{\mathbf{k}}(\mathbf{m}-\rho)\Phi_{\mathbf{m}}(a_{1}) \nonumber \\
&=(1-a_{1})^{-r|\alpha|-dr(r-1)}\sum_{\mathbf{k} \in \mathscr{P}}d_{\mathbf{k}}\frac{1}{\left(\frac{n}{r}\right)_{\mathbf{k}}}\Phi_{\mathbf{k}}\left(\frac{a_{1}}{1-a_{1}}u\right) \nonumber \\
\label{eq:absolute conv for MLP}
&=(1-a_{1})^{-r|\alpha|-dr(r-1)}e^{\frac{a_{1}}{1-a_{1}}\tr{u}}< \infty. 
\end{align}
\end{proof}


Let us consider the operators $D_{\alpha}^{(j)}$ for $j=1,2,3$. 
The operator $D_{\alpha}^{(3)}$ is a first order differential operator on the domain $\mathcal{D}$:
\begin{equation}
D_{\alpha}^{(3)}:=2\tr{(w\nabla_{w})}.
\end{equation}
Since this is the Euler operator, 
$$
D_{\alpha}^{(3)}f^{(\alpha)}_{\mathbf{m}}(w)=2|\mathbf{m}|f^{(\alpha)}_{\mathbf{m}}(w).
$$
The operators $D_{\alpha}^{(2)}$ and $D_{\alpha}^{(1)}$ are respectively defined by 
$C_{\alpha}^{-1}D_{\alpha}^{(2)}=D_{\alpha}^{(3)}C_{\alpha}^{-1}$ and $\mathcal{L}_{\alpha}D_{\alpha}^{(1)}=D_{\alpha}^{(2)}\mathcal{L}_{\alpha}$. 
Hence, $D_{\alpha}^{(2)}F^{(\alpha)}_{\mathbf{m}}(w)=2|\mathbf{m}|F^{(\alpha)}_{\mathbf{m}}(w)$ and 
\begin{equation}
\label{eq:differential equation for Laguerre}
D_{\alpha}^{(1)}\psi_{\mathbf{m}}^{(\alpha)}(u)=2|\mathbf{m}|\psi_{\mathbf{m}}^{(\alpha)}(u).
\end{equation}
Moreover, they have the following expressions.
\begin{align}
D_{\alpha}^{(2)}&=\tr{((z^{2}-e)\nabla_{z}+\alpha(z-e))}, \\
D_{\alpha}^{(1)}&=\tr{(-u\nabla_{u}^{2}-\alpha\nabla_{u}+u-\alpha{e})}.
\end{align}

\begin{lem}
\label{thm:differential Lemma for Laguerre and difference Lemma for Meixner}
{\rm{(1)}}\,
\begin{align}
D_{\alpha}^{(1)}\varphi_{\mathbf{s}}(u)
&=\sum_{j=1}^{r}a_{j}(\mathbf{s})\varphi_{\mathbf{s}+\epsilon_{j}}(u)-r\alpha\varphi_{\mathbf{s}}(u)   \nonumber \\ 
\label{eq:differntial of D3 and spherical fnc}
{} & \quad -\sum_{j=1}^{r}\left(s_{j}+\frac{d}{4}(r-1)\right)\left(s_{j}+\alpha- \frac{d}{4}(r-1)-1\right)a_{j}(-\mathbf{s})\varphi_{\mathbf{s}-\epsilon_{j}}(u).
\end{align}
{\rm{(2)}}\,
\begin{align}
D_{\alpha}^{(1)}\Phi_{\mathbf{x}}(u)
&=\sum_{j=1}^{r}\tilde{a_{j}}(\mathbf{x})\Phi_{\mathbf{x}+\epsilon_{j}}(u)-r\alpha\Phi_{\mathbf{x}}(u)   \nonumber \\ 
\label{eq:differntial of D3 and spherical poly}
{} & \quad -\sum_{j=1}^{r}\left(x_{j}+\frac{d}{2}(r-j)\right)\left(x_{j}+\alpha-1- \frac{d}{2}(j-1)\right)\tilde{a_{j}}(-\mathbf{x})\Phi_{\mathbf{x}-\epsilon_{j}}(u).
\end{align}
Here, 
\begin{equation}
\tilde{a_{j}}(\mathbf{x}):=a_{j}(\mathbf{x}-\rho)=\prod_{k\not=j}\frac{x_{j}-x_{k}-\frac{d}{2}(j-k-1)}{x_{j}-x_{k}-\frac{d}{2}(j-k)}.
\end{equation}
{\rm{(3)}}\,For any $C \in \mathbb{C}$, 
\begin{align}
e^{C\tr{u}}D_{\alpha}^{(1)}e^{-C\tr{u}}\Phi_{\mathbf{x}}(u)
&=(1-C^{2})\sum_{j=1}^{r}\tilde{a_{j}}(\mathbf{x})\Phi_{\mathbf{x}+\epsilon_{j}}(u)   \nonumber \\ 
{} & \quad +\sum_{j=1}^{r}(C(2x_{j}+\alpha)-\alpha)\Phi_{\mathbf{x}}(u) \nonumber \\
\label{eq:conj differntial of D3 and spherical poly}
{} & \quad -\sum_{j=1}^{r}\left(x_{j}+\frac{d}{2}(r-j)\right)\left(x_{j}+\alpha- 1-\frac{d}{2}(j-1)\right)\tilde{a_{j}}(-\mathbf{x})\Phi_{\mathbf{x}-\epsilon_{j}}(u).
\end{align}
\end{lem}
\begin{rmk}
Though (\ref{eq:differntial of D3 and spherical fnc}) is a corollary of Lemma\,5.8 in \cite{FW1} (version.$1$) essentially, 
Faraut and Wakayama's lemma is deleted in their revised version. 
Hence, we re-prove it according to their proof. 
\end{rmk}
\begin{proof}
{\rm{(1)}}\,
The modified Laplace transform of $\varphi_{\mathbf{s}}$ is given by
\begin{align}
(\mathcal{L}_{\alpha}\varphi_{\mathbf{s}})(z)
&=\frac{2^{r\alpha}}{\Gamma_{\Omega}(\alpha)}\int_{\Omega}e^{-(z|u)}\varphi_{\mathbf{s}}(u)\Delta(u)^{\alpha -\frac{n}{r}}\,du 
=2^{r\alpha}\frac{\Gamma_{\Omega}(\mathbf{s}+\alpha+\rho)}{\Gamma_{\Omega}(\alpha)}\varphi_{-\mathbf{s}-\alpha}(z). \nonumber 
\end{align}
Thus, from the definition of $D_{\alpha}^{(1)}$ and Lemma\,\ref{thm:Pieri},
\begin{align}
\mathcal{L}_{\alpha}(D_{\alpha}^{(1)}\varphi_{\mathbf{s}})(z)
&=D_{\alpha}^{(2)}(\mathcal{L}_{\alpha}\varphi_{\mathbf{s}})(z) \nonumber \\
&=2^{r\alpha}\frac{\Gamma_{\Omega}(\mathbf{s}+\alpha+\rho)}{\Gamma_{\Omega}(\alpha)}D_{\alpha}^{(2)}\varphi_{-\mathbf{s}-\alpha}(z) \nonumber \\
&=\sum_{j=1}^{r}\left(s_{j}+\alpha-\frac{d}{4}(r-1)\right)a_{j}(\mathbf{s}+\alpha)2^{r\alpha}\frac{\Gamma_{\Omega}(\mathbf{s}+\alpha+\rho)}{\Gamma_{\Omega}(\alpha)}\varphi_{-\mathbf{s}-\alpha-\epsilon_{j}}(z) \nonumber \\
{} & \quad -r\alpha2^{r\alpha}\frac{\Gamma_{\Omega}(\mathbf{s}+\alpha+\rho)}{\Gamma_{\Omega}(\alpha)}\varphi_{-\mathbf{s}-\alpha}(z) \nonumber \\
{} & \quad -\sum_{j=1}^{r}\left(s_{j}+\frac{d}{4}(r-1)\right)a_{j}(-\mathbf{s}-\alpha)2^{r\alpha}\frac{\Gamma_{\Omega}(\mathbf{s}+\alpha+\rho)}{\Gamma_{\Omega}(\alpha)}\varphi_{-\mathbf{s}-\alpha+\epsilon_{j}}(z). \nonumber
\end{align}
Since
$$
\varphi_{-\mathbf{s}-\alpha\pm \epsilon_{j}}(z)=\frac{\Gamma_{\Omega}(\mathbf{s}+\alpha+\rho)}{\Gamma_{\Omega}(\mathbf{s}+\alpha+\rho\pm \epsilon_{j})}\mathcal{L}_{\alpha}(\varphi_{\mathbf{s}\pm \epsilon_{j}})(z), 
$$
we have
\begin{align}
D_{\alpha}^{(2)}(\mathcal{L}_{\alpha}\varphi_{\mathbf{s}})(z)
&=\sum_{j=1}^{r}\left(s_{j}+\alpha-\frac{d}{4}(r-1)\right)a_{j}(\mathbf{s})\frac{\Gamma_{\Omega}(\mathbf{s}+\alpha+\rho)}{\Gamma_{\Omega}(\mathbf{s}+\alpha+\rho+\epsilon_{j})}
\mathcal{L}_{\alpha}(\varphi_{\mathbf{s}+\epsilon_{j}})(z) \nonumber \\
{} & \quad -r\alpha(\mathcal{L}_{\alpha}\varphi_{\mathbf{s}})(z) \nonumber \\
{} & \quad -\sum_{j=1}^{r}\left(s_{j}+\frac{d}{4}(r-1)\right)a_{j}(-\mathbf{s})\frac{\Gamma_{\Omega}(\mathbf{s}+\alpha+\rho)}{\Gamma_{\Omega}(\mathbf{s}+\alpha+\rho-\epsilon_{j})}
\mathcal{L}_{\alpha}(\varphi_{\mathbf{s}-\epsilon_{j}})(z) \nonumber \\
&=\mathcal{L}_{\alpha}
\left(\sum_{j=1}^{r}a_{j}(\mathbf{s})\varphi_{\mathbf{s}+\epsilon_{j}}(u)-r\alpha\varphi_{\mathbf{s}}(u) \right.  \nonumber \\ 
{} & \quad \left. -\sum_{j=1}^{r}\left(s_{j}+\frac{d}{4}(r-1)\right)\left(s_{j}+\alpha- \frac{d}{4}(r-1)-1\right)a_{j}(-\mathbf{s})\varphi_{\mathbf{s}-\epsilon_{j}}(u)\right)(z). \nonumber
\end{align}

\noindent
{\rm{(2)}}\,
Put $\mathbf{s}=\mathbf{m}-\rho$ in (\ref{eq:differntial of D3 and spherical fnc}).

\noindent
{\rm{(3)}}\,
By
$$
e^{C\tr{u}}\nabla_{u}e^{-C\tr{u}}=-Ce,\,\,\,\,\,e^{C\tr{u}}\tr{(u\nabla_{u}^{2})}e^{-C\tr{u}}=C^{2}\tr{u}
$$
($e$ in $-Ce$ is the unit element of $V$) and the product rule of differentiation, 
we remark that 
\begin{align}
e^{C\tr{u}}\tr{(u\nabla_{u}^{2})}e^{-C\tr{u}}\Phi_{\mathbf{x}}(u)
&=\tr{(u\nabla_{u}^{2})}\Phi_{\mathbf{x}}(u) \nonumber \\
{} & \quad +2\tr{(ue^{C\tr{u}}\nabla_{u}(e^{-C\tr{u}})\nabla_{u}(\Phi_{\mathbf{x}}(u)))} \nonumber \\
{} & \quad +e^{C\tr{u}}\Phi_{\mathbf{x}}(u)\tr{(u\nabla_{u}^{2})}e^{-C\tr{u}} \nonumber \\
&=\left\{\tr{(u\nabla_{u}^{2})}+C^{2}\tr{u}-2C|\mathbf{x}|\right\}\Phi_{\mathbf{x}}(u) \nonumber
\end{align}
and 
\begin{align}
e^{C\tr{u}}\tr{(\nabla_{u})}e^{-C\tr{u}}\Phi_{\mathbf{x}}(u)
&=\Phi_{\mathbf{x}}(u)\tr{(e^{C\tr{u}}\nabla_{u}e^{-C\tr{u}})}+\tr{(\nabla_{u})}\Phi_{\mathbf{x}}(u) \nonumber \\
&=-Cr\Phi_{\mathbf{x}}(u)+\tr{(\nabla_{u})}\Phi_{\mathbf{x}}(u). \nonumber
\end{align}
Hence, 
\begin{align}
e^{C\tr{u}}D_{\alpha}^{(1)}e^{-C\tr{u}}\Phi_{\mathbf{x}}(u)
&=D_{\alpha}^{(1)}\Phi_{\mathbf{x}}(u)-C^{2}\tr{u}\Phi_{\mathbf{x}}(u)+C(2|\mathbf{x}|+r\alpha)\Phi_{\mathbf{x}}(u). \nonumber
\end{align}
Therefore, from (\ref{eq:differntial of D3 and spherical poly}) and (\ref{eq:Pieri 1}), 
\begin{align}
e^{C\tr{u}}D_{\alpha}^{(1)}e^{-C\tr{u}}\Phi_{\mathbf{x}}(u)
&=\sum_{j=1}^{r}\tilde{a_{j}}(\mathbf{x})\Phi_{\mathbf{x}+\epsilon_{j}}(u)-r\alpha\Phi_{\mathbf{x}}(u)   \nonumber \\ 
{} & \quad -\sum_{j=1}^{r}\left(x_{j}+\frac{d}{2}(r-j)\right)\left(x_{j}+\alpha-1- \frac{d}{2}(j-1)\right)\tilde{a_{j}}(-\mathbf{x})\Phi_{\mathbf{x}-\epsilon_{j}}(u) \nonumber \\
{} & \quad -C^{2}\sum_{j=1}^{r}\tilde{a_{j}}(\mathbf{x})\Phi_{\mathbf{x}+\epsilon_{j}}(u)+C(2|\mathbf{x}|+r\alpha)\Phi_{\mathbf{x}}(u) \nonumber \\
&=(1-C^{2})\sum_{j=1}^{r}\tilde{a_{j}}(\mathbf{x})\Phi_{\mathbf{x}+\epsilon_{j}}(u)   \nonumber \\ 
{} & \quad +\sum_{j=1}^{r}(C(2x_{j}+\alpha)-\alpha)\Phi_{\mathbf{x}}(u) \nonumber \\
{} & \quad -\sum_{j=1}^{r}\left(x_{j}+\frac{d}{2}(r-j)\right)\left(x_{j}+\alpha- 1-\frac{d}{2}(j-1)\right)\tilde{a_{j}}(-\mathbf{x})\Phi_{\mathbf{x}-\epsilon_{j}}(u). \nonumber 
\end{align}
\end{proof}

\section{Multivariate Meixner, Charlier and Krawtchouk polynomials}

In this section, we assume that $\mathbf{m}, \mathbf{n}, \mathbf{x}, \mathbf{y} \in \mathscr{P}, \alpha \in \mathbb{C}, c,a,p \in \mathbb{C}^{\ast}, N \in \mathbb{Z}_{\geq 0}$ and 
$$
z=u_{1}\sum_{j=1}^{r}a_{j}c_{j}, \,\,\,
w=u_{2}\sum_{j=1}^{r}b_{j}c_{j} \in V^{\mathbb{C}},
$$
with $u_{1},u_{2} \in U$, $a_{1}\geq \cdots\geq a_{r}\geq 0$, $b_{1}\geq \cdots\geq b_{r}\geq 0$ unless otherwise specified.

\subsection{Definitions}

\begin{dfn}
\label{thm:def of the MDOP}
\noindent
We define the multivariate Meixner, Charlier and Krawtchouk polynomials as follows.
\begin{align}
M_{\mathbf{m}}(\mathbf{x};\alpha ,c)
:=&\sum_{\mathbf{k}\subset \mathbf{m}}\frac{1}{d_{\mathbf{k}}}\frac{\left(\frac{n}{r}\right)_{\mathbf{k}}}{(\alpha)_{\mathbf{k}}}\binom{\mathbf{m}}{\mathbf{k}}\binom{\mathbf{x}}{\mathbf{k}}\left(1-\frac{1}{c}\right)^{|\mathbf{k}|} \\
=&\sum_{\mathbf{k}\subset \mathbf{m}}\binom{\mathbf{m}}{\mathbf{k}}\frac{\gamma_{\mathbf{k}}(\mathbf{x}-\rho)}{(\alpha)_{\mathbf{k}}}\left(1-\frac{1}{c}\right)^{|\mathbf{k}|} \\
=&\sum_{\mathbf{k}\subset \mathbf{m}}d_{\mathbf{k}}\frac{\gamma_{\mathbf{k}}(\mathbf{m}-\rho)\gamma_{\mathbf{k}}(\mathbf{x}-\rho)}{\left(\frac{n}{r}\right)_{\mathbf{k}}(\alpha)_{\mathbf{k}}}\left(1-\frac{1}{c}\right)^{|\mathbf{k}|}, \\
C_{\mathbf{m}}(\mathbf{x};a):=&\sum_{\mathbf{k}\subset \mathbf{m}}\frac{1}{d_{\mathbf{k}}}\left(\frac{n}{r}\right)_{\mathbf{k}}\binom{\mathbf{m}}{\mathbf{k}}\binom{\mathbf{x}}{\mathbf{k}}\left(-\frac{1}{a}\right)^{|\mathbf{k}|} \\
=&\sum_{\mathbf{k}\subset \mathbf{m}}\binom{\mathbf{m}}{\mathbf{k}}\gamma_{\mathbf{k}}(\mathbf{x}-\rho)\left(-\frac{1}{a}\right)^{|\mathbf{k}|} \\
=&\sum_{\mathbf{k}\subset \mathbf{m}}d_{\mathbf{k}}\frac{\gamma_{\mathbf{k}}(\mathbf{m}-\rho)\gamma_{\mathbf{k}}(\mathbf{x}-\rho)}{\left(\frac{n}{r}\right)_{\mathbf{k}}}\left(-\frac{1}{a}\right)^{|\mathbf{k}|}, \\
K_{\mathbf{m}}(\mathbf{x};p ,N):=&\sum_{\mathbf{k}\subset \mathbf{m}}\frac{1}{d_{\mathbf{k}}}\frac{\left(\frac{n}{r}\right)_{\mathbf{k}}}{(-N)_{\mathbf{k}}}\binom{\mathbf{m}}{\mathbf{k}}\binom{\mathbf{x}}{\mathbf{k}}\left(\frac{1}{p}\right)^{|\mathbf{k}|}\,\,\,(\mathbf{m} \subset N=(N,\ldots,N)) \\
=&\sum_{\mathbf{k}\subset \mathbf{m}}\binom{\mathbf{m}}{\mathbf{k}}\frac{\gamma_{\mathbf{k}}(\mathbf{x}-\rho)}{(-N)_{\mathbf{k}}}\left(\frac{1}{p}\right)^{|\mathbf{k}|} \\
=&\sum_{\mathbf{k}\subset \mathbf{m}}d_{\mathbf{k}}\frac{\gamma_{\mathbf{k}}(\mathbf{m}-\rho)\gamma_{\mathbf{k}}(\mathbf{x}-\rho)}{\left(\frac{n}{r}\right)_{\mathbf{k}}(-N)_{\mathbf{k}}}\left(\frac{1}{p}\right)^{|\mathbf{k}|}.
\end{align}
\end{dfn}
When $r=1$, these polynomials become the usual Meixner, Charlier and Krawtchouk polynomials. 
By the definition, we immediately obtain a duality property for these polynomials. 
\begin{prop}
\label{thm:duality of MDOP}
{\rm{(1)}}\,For all $\mathbf{m}, \mathbf{x} \in \mathscr{P}$, we have
\begin{equation}
M_{\mathbf{m}}(\mathbf{x};\alpha ,c)=M_{\mathbf{x}}(\mathbf{m};\alpha ,c).
\end{equation}
{\rm{(2)}}\,For all $\mathbf{m}, \mathbf{x} \in \mathscr{P}$, we have
\begin{equation}
C_{\mathbf{m}}(\mathbf{x};a)=C_{\mathbf{x}}(\mathbf{m};a).
\end{equation}
{\rm{(3)}}\,For all $\mathbf{m}, \mathbf{x} \subset N$, we have
\begin{equation}
\label{eq:duality of MK}
K_{\mathbf{m}}(\mathbf{x};p ,N)=K_{\mathbf{x}}(\mathbf{m};p ,N).
\end{equation}
\end{prop}
We only remark the proof of (\ref{eq:duality of MK}). Since $\mathbf{m}, \mathbf{x} \subset N$, 
\begin{align}
K_{\mathbf{m}}(\mathbf{x};p ,N)&=\sum_{\mathbf{k}\subset \mathbf{m}}\frac{1}{d_{\mathbf{k}}}\frac{\left(\frac{n}{r}\right)_{\mathbf{k}}}{(-N)_{\mathbf{k}}}\binom{\mathbf{m}}{\mathbf{k}}\binom{\mathbf{x}}{\mathbf{k}}\left(\frac{1}{p}\right)^{|\mathbf{k}|} \nonumber \\
&=\sum_{\mathbf{k}\subset \mathbf{x}}\frac{1}{d_{\mathbf{k}}}\frac{\left(\frac{n}{r}\right)_{\mathbf{k}}}{(-N)_{\mathbf{k}}}\binom{\mathbf{x}}{\mathbf{k}}\binom{\mathbf{m}}{\mathbf{k}}\left(\frac{1}{p}\right)^{|\mathbf{k}|}=K_{\mathbf{x}}(\mathbf{m};p ,N). \nonumber
\end{align}
We also obtain the following relations by the definitions. 
\begin{prop}
\label{thm:relation of MDOP}
{\rm{(1)}}\,
\begin{equation}
\label{eq:relation Meixner and Krawtchouk}
M_{\mathbf{m}}\left(\mathbf{x};-N ,\frac{p}{p-1}\right)=K_{\mathbf{m}}(\mathbf{x};p ,N).
\end{equation}
{\rm{(2)}}\,
\begin{equation}
\label{eq:relation Meixner and Charlier}
\lim_{\alpha \to \infty}M_{\mathbf{m}}\left(\mathbf{x};\alpha ,\frac{a}{a+\alpha}\right)=C_{\mathbf{m}}(\mathbf{x};a).
\end{equation}
{\rm{(3)}}\,
\begin{equation}
\label{eq:relation Krawtchouk and Charlier}
\lim_{N \to \infty}K_{\mathbf{m}}\left(\mathbf{x};\frac{a}{N} , N\right)=C_{\mathbf{m}}(\mathbf{x};a).
\end{equation}
\end{prop}
Actually, {\rm{(1)}} follows from the definitions. 
For {\rm{(2)}} and {\rm{(3)}}, we remark that
\begin{align}
\lim_{\alpha \to \infty}\frac{\alpha^{|\mathbf{k}|}}{(\alpha)_{\mathbf{k}}}
&=\lim_{\alpha \to \infty}\prod_{j=1}^{r}\frac{\alpha^{k_{j}}}{\left(\alpha -\frac{d}{2}(j-1)\right)_{k_{j}}}=1,\nonumber \\
\lim_{N \to \infty}\frac{N^{|\mathbf{k}|}}{(-N)_{\mathbf{k}}}
&=\lim_{N \to \infty}\prod_{j=1}^{r}\frac{N^{k_{j}}}{\left(-N -\frac{d}{2}(j-1)\right)_{k_{j}}}=(-1)^{|\mathbf{k}|}.\nonumber
\end{align}

\subsection{Generating functions}

To present our key lemma which is a summation formula of the above polynomials, we need to prove their convergence.  
\begin{lem}
\label{thm:absolute convergence of gen fnc of gen fnc for Meixner}

\noindent
{\rm{(1)}}\,
If $1>a_{1}\geq \cdots\geq a_{r}\geq 0$, $b_{1}\geq \cdots\geq b_{r}\geq 0$, then 
\begin{equation}
\sum_{\mathbf{x},\mathbf{m} \in \mathscr{P}}\left|d_{\mathbf{m}}\frac{(\alpha)_{\mathbf{m}}}{\left(\frac{n}{r}\right)_{\mathbf{m}}}M_{\mathbf{m}}(\mathbf{x};\alpha ,c)\Phi_{\mathbf{m}}(z)
d_{\mathbf{x}}\frac{1}{\left(\frac{n}{r}\right)_{\mathbf{x}}}\Phi_{\mathbf{x}}(w)\right|
\leq e^{r b_{1}\left(1+\frac{a_{1}}{1-a_{1}}\left(\frac{1}{c}-1\right)\right)}(1-a_{1})^{-r(|\alpha|+2n)}.
\end{equation}
{\rm{(2)}}\,For any $z,w \in V^{\mathbb{C}}$, we have
\begin{equation}
\sum_{\mathbf{x},\mathbf{m} \in \mathscr{P}}\left|d_{\mathbf{m}}\frac{1}{\left(\frac{n}{r}\right)_{\mathbf{m}}}C_{\mathbf{m}}(\mathbf{x},a)\Phi_{\mathbf{m}}(z)
d_{\mathbf{x}}\frac{1}{\left(\frac{n}{r}\right)_{\mathbf{x}}}\Phi_{\mathbf{x}}(w)\right|
\leq e^{r \left(a_{1}+b_{1}+\frac{a_{1}b_{1}}{a}\right)}.
\end{equation}
\end{lem}
\begin{proof}
{\rm{(1)}}\,
By Lemma\,\ref{thm:ineq for generalized shifted factorial}, Lemma\,\ref{thm:FK,Thm12.1.1} and Lemma\,\ref{thm:positivity of shifted Jack}, 
\begin{align}
&\sum_{\mathbf{x},\mathbf{m} \in \mathscr{P}}\left|d_{\mathbf{m}}\frac{(\alpha)_{\mathbf{m}}}{\left(\frac{n}{r}\right)_{\mathbf{m}}}M_{\mathbf{m}}(\mathbf{x};\alpha ,c)\Phi_{\mathbf{m}}(z)
d_{\mathbf{x}}\frac{1}{\left(\frac{n}{r}\right)_{\mathbf{x}}}\Phi_{\mathbf{x}}(w)\right| \nonumber \\
&\leq \sum_{\mathbf{k}\in \mathscr{P}}d_{\mathbf{k}}\frac{1}{\left(\frac{n}{r}\right)_{\mathbf{k}}(|\alpha|+d(r-1))_{\mathbf{k}}}\left(\frac{1}{c}-1\right)^{|\mathbf{k}|} \nonumber \\
{} & \quad \cdot 
\sum_{\mathbf{m}\in \mathscr{P}}d_{\mathbf{m}}\frac{(|\alpha| +d(r-1))_{\mathbf{m}}}{\left(\frac{n}{r}\right)_{\mathbf{m}}}\gamma_{\mathbf{k}}(\mathbf{m}-\rho)\Phi_{\mathbf{m}}(a_{1}) 
\sum_{\mathbf{x} \in \mathscr{P}}d_{\mathbf{x}}\frac{1}{\left(\frac{n}{r}\right)_{\mathbf{x}}}\gamma_{\mathbf{k}}(\mathbf{x}-\rho)\Phi_{\mathbf{x}}(b_{1}). \nonumber
\end{align}
Moreover, from (\ref{eq:basic expansion 1}) and (\ref{eq:basic expansion 2}) of Theorem\,\ref{thm:spherical Taylor expan lem}, 
\begin{align}
\sum_{\mathbf{m}\in \mathscr{P}}d_{\mathbf{m}}\frac{(|\alpha| +d(r-1))_{\mathbf{m}}}{\left(\frac{n}{r}\right)_{\mathbf{m}}}\gamma_{\mathbf{k}}(\mathbf{m}-\rho)\Phi_{\mathbf{m}}(a_{1})
&=(|\alpha| +d(r-1))_{\mathbf{k}}(1-a_{1})^{-r|\alpha| -dr(r-1)-|\mathbf{k}|}
a_{1}^{|\mathbf{k}|}, \nonumber \\
\sum_{\mathbf{x} \in \mathscr{P}}d_{\mathbf{x}}\frac{1}{\left(\frac{n}{r}\right)_{\mathbf{x}}}\gamma_{\mathbf{k}}(\mathbf{x}-\rho)\Phi_{\mathbf{x}}(b_{1})
&=e^{rb_{1}}
{b_{1}}^{|\mathbf{k}|}. \nonumber 
\end{align}
Therefore, we have 
\begin{align}
&\sum_{\mathbf{x},\mathbf{m} \in \mathscr{P}}\left|d_{\mathbf{m}}\frac{(\alpha)_{\mathbf{m}}}{\left(\frac{n}{r}\right)_{\mathbf{m}}}M_{\mathbf{m}}(\mathbf{x};\alpha ,c)\Phi_{\mathbf{m}}(z)
d_{\mathbf{x}}\frac{1}{\left(\frac{n}{r}\right)_{\mathbf{x}}}\Phi_{\mathbf{x}}(w)\right| \nonumber \\
&\leq e^{rb_{1}}(1-a_{1})^{-r(|\alpha|+d(r-1))}\sum_{\mathbf{k}\in \mathscr{P}}d_{\mathbf{k}}\frac{1}{\left(\frac{n}{r}\right)_{\mathbf{k}}}\Phi_{\mathbf{k}}\left(\left(\frac{1}{c}-1\right)\frac{a_{1}b_{1}}{1-a_{1}}\right) \nonumber \\
&=e^{r b_{1}\left(1+\frac{a_{1}}{1-a_{1}}\left(\frac{1}{c}-1\right)\right)}(1-a_{1})^{-r(|\alpha|+d(r-1))}<\infty. \nonumber 
\end{align}

\noindent
{\rm{(2)}}\,By a similar argument, 
\begin{align}
&\sum_{\mathbf{x},\mathbf{m} \in \mathscr{P}}\left|d_{\mathbf{m}}\frac{1}{\left(\frac{n}{r}\right)_{\mathbf{m}}}C_{\mathbf{m}}(\mathbf{x},a)\Phi_{\mathbf{m}}(z)
d_{\mathbf{x}}\frac{1}{\left(\frac{n}{r}\right)_{\mathbf{x}}}\Phi_{\mathbf{x}}(w)\right| \nonumber \\
&\leq \sum_{\mathbf{k}\in \mathscr{P}}d_{\mathbf{k}}\frac{1}{\left(\frac{n}{r}\right)_{\mathbf{k}}}a^{-|\mathbf{k}|}
\sum_{\mathbf{m}\in \mathscr{P}}d_{\mathbf{m}}\frac{1}{\left(\frac{n}{r}\right)_{\mathbf{m}}}\gamma_{\mathbf{k}}(\mathbf{m}-\rho)\Phi_{\mathbf{m}}(a_{1})
\sum_{\mathbf{x} \in \mathscr{P}}d_{\mathbf{x}}\frac{1}{\left(\frac{n}{r}\right)_{\mathbf{x}}}\gamma_{\mathbf{k}}(\mathbf{x}-\rho)\Phi_{\mathbf{x}}(b_{1}) \nonumber \\
&=e^{r(a_{1}+b_{1})}\sum_{\mathbf{k}\in \mathscr{P}}d_{\mathbf{k}}\frac{1}{\left(\frac{n}{r}\right)_{\mathbf{k}}}\Phi_{\mathbf{k}}\left(\frac{a_{1}b_{1}}{a}\right) \nonumber \\
&=e^{r \left(a_{1}+b_{1}+\frac{a_{1}b_{1}}{a}\right)}<\infty. \nonumber
\end{align}
\end{proof}
The following theorem is the key result in our theory.  
\begin{thm}
\label{thm:master generating fnc 1}
{\rm{(1)}}\,
For $z \in \mathcal{D}, w \in V^{\mathbb{C}}$, we obtain
\begin{align}
\label{eq:master gene fnc 1.1}
\sum_{\mathbf{m} \in \mathscr{P}}e^{\tr{w}}L_{\mathbf{m}}^{\left(\alpha -\frac{n}{r}\right)}\left(\left(\frac{1}{c}-1\right)w\right)\Phi_{\mathbf{m}}(z)
&=\sum_{\mathbf{x},\mathbf{m} \in \mathscr{P}}d_{\mathbf{m}}\frac{(\alpha)_{\mathbf{m}}}{\left(\frac{n}{r}\right)_{\mathbf{m}}}M_{\mathbf{m}}(\mathbf{x};\alpha ,c)\Phi_{\mathbf{m}}(z) \nonumber \\
{} & \quad \cdot d_{\mathbf{x}}\frac{1}{\left(\frac{n}{r}\right)_{\mathbf{x}}}\Phi_{\mathbf{x}}(w) \\
\label{eq:master gene fnc 1.2}
&=\Delta(e-z)^{-\alpha}\int_{K}e^{(kw|(e-\frac{1}{c}z)(e-z)^{-1})}\,dk. 
\end{align}
{\rm{(2)}}\,For $w,z \in V^{\mathbb{C}}$, we obtain
\begin{align}
\label{eq:master gene fnc 2.1}
\sum_{\mathbf{m} \in \mathscr{P}}d_{\mathbf{m}}\frac{1}{\left(\frac{n}{r}\right)_{\mathbf{m}}}e^{\tr{w}}\Phi_{\mathbf{m}}\left(e-\frac{1}{a}w\right)\Phi_{\mathbf{m}}(z)
&=\sum_{\mathbf{x},\mathbf{m} \in \mathscr{P}}d_{\mathbf{m}}\frac{1}{\left(\frac{n}{r}\right)_{\mathbf{m}}}C_{\mathbf{m}}(\mathbf{x};a)\Phi_{\mathbf{m}}(z) \nonumber \\
{} & \quad \cdot d_{\mathbf{x}}\frac{1}{\left(\frac{n}{r}\right)_{\mathbf{x}}}\Phi_{\mathbf{x}}(w) \\
\label{eq:master gene fnc 2.2}
&=e^{\tr{(w+z)}}\int_{K}e^{-\frac{1}{a}(kw|z)}\,dk. 
\end{align}
\end{thm}
\begin{rmk}
We remark for any $\mathbf{m},\mathbf{x} \in \mathscr{P}$ and $\alpha \in \mathbb{C}$, 
$$
L_{\mathbf{m}}^{\left(\alpha -\frac{n}{r}\right)}(0)
=d_{\mathbf{m}}\frac{(\alpha)_{\mathbf{m}}}{\left(\frac{n}{r}\right)_{\mathbf{m}}}
$$
and $M_{\mathbf{m}}(\mathbf{x};\alpha ,1)=1$. 
Therefore, for $c=1$, (\ref{eq:master gene fnc 1.1}) is trivial and (\ref{eq:master gene fnc 1.2}) degenerates to 
\begin{equation}
\sum_{\mathbf{m} \in \mathscr{P}}d_{\mathbf{m}}\frac{(\alpha)_{\mathbf{m}}}{\left(\frac{n}{r}\right)_{\mathbf{m}}}\Phi_{\mathbf{m}}(z)
\cdot\sum_{\mathbf{x} \in \mathscr{P}}d_{\mathbf{x}}\frac{1}{\left(\frac{n}{r}\right)_{\mathbf{x}}}\Phi_{\mathbf{x}}(w) 
=\Delta(e-z)^{-\alpha}e^{\tr{w}}, \nonumber
\end{equation}
which is well known formula. 
\end{rmk}
\begin{proof}
{\rm{(1)}}\,
By the above lemma, the series converges absolutely under the conditions. Therefore, we derive
\begin{align}
&\sum_{\mathbf{x},\mathbf{m} \in \mathscr{P}}d_{\mathbf{m}}\frac{(\alpha)_{\mathbf{m}}}{\left(\frac{n}{r}\right)_{\mathbf{m}}}M_{\mathbf{m}}(\mathbf{x};\alpha ,c)\Phi_{\mathbf{m}}(z)
d_{\mathbf{x}}\frac{1}{\left(\frac{n}{r}\right)_{\mathbf{x}}}\Phi_{\mathbf{x}}(w) \nonumber \\
&= \sum_{\mathbf{m}\in \mathscr{P}}d_{\mathbf{m}}\frac{(\alpha)_{\mathbf{m}}}{\left(\frac{n}{r}\right)_{\mathbf{m}}}\Phi_{\mathbf{m}}(z)
\sum_{\mathbf{k}\subset \mathbf{m}}\binom{\mathbf{m}}{\mathbf{k}}\frac{1}{(\alpha)_{\mathbf{k}}}\left(1-\frac{1}{c}\right)^{|\mathbf{k}|} 
\sum_{\mathbf{x} \in \mathscr{P}}d_{\mathbf{x}}\frac{1}{\left(\frac{n}{r}\right)_{\mathbf{x}}}\gamma_{\mathbf{k}}(\mathbf{x}-\rho)\Phi_{\mathbf{x}}(w) \nonumber \\
&= \sum_{\mathbf{m}\in \mathscr{P}}e^{\tr{w}}d_{\mathbf{m}}\frac{(\alpha)_{\mathbf{m}}}{\left(\frac{n}{r}\right)_{\mathbf{m}}}
\sum_{\mathbf{k}\subset \mathbf{m}}\binom{\mathbf{m}}{\mathbf{k}}\frac{(-1)^{\mathbf{k}}}{(\alpha)_{\mathbf{k}}}\Phi_{\mathbf{k}}\left(\left(\frac{1}{c}-1\right)w\right)\Phi_{\mathbf{m}}(z) \nonumber \\
&=\sum_{\mathbf{m} \in \mathscr{P}}e^{\tr{w}}L_{\mathbf{m}}^{\left(\alpha -\frac{n}{r}\right)}\left(\left(\frac{1}{c}-1\right)w\right)\Phi_{\mathbf{m}}(z). \nonumber
\end{align}
(\ref{eq:master gene fnc 1.2}) follows from (\ref{eq:generating fnc of Laguerre}).

\noindent
{\rm{(2)}}\,Put $c=\frac{a}{a+\alpha}, w \to \frac{w}{\alpha}, a,\alpha \in \mathbb{R}_{>0}$ in {\rm{(1)}} of Theorem\,\ref{thm:master generating fnc 1} and take the limit of $\alpha \to \infty$. 
\end{proof}


The generating formulas of our polynomials are a corollary of the above theorem. 
\begin{thm}
\label{thm:generating fnc of MDOP}
{\rm{(1)}}\,
For $z \in \mathcal{D}, \mathbf{x} \in \mathscr{P}$, we have
\begin{equation}
\label{eq:gen fnc of Meixner}
\Delta(e-z)^{-\alpha}\Phi_{\mathbf{x}}\left(\left(e-\frac{1}{c}z\right)(e-z)^{-1}\right)
=\sum_{\mathbf{n} \in \mathscr{P}}d_{\mathbf{n}}\frac{(\alpha)_{\mathbf{n}}}{\left(\frac{n}{r}\right)_{\mathbf{n}}}M_{\mathbf{n}}(\mathbf{x};\alpha ,c)\Phi_{\mathbf{n}}(z). 
\end{equation}
{\rm{(2)}}\,
For $z \in \mathcal{D}, \mathbf{x} \in \mathscr{P}$, we have
\begin{equation}
\label{eq:gen fnc of Charlier}
e^{\tr{z}}\Phi_{\mathbf{x}}\left(e-\frac{1}{a}z\right)=\sum_{\mathbf{n} \in \mathscr{P}}d_{\mathbf{n}}\frac{1}{\left(\frac{n}{r}\right)_{\mathbf{n}}}C_{\mathbf{n}}(\mathbf{x};a)\Phi_{\mathbf{n}}(z).
\end{equation}
{\rm{(3)}}\,
For $z \in \mathcal{D}, \mathbf{x} \subset N$, we have
\begin{equation}
\label{eq:gen fnc of Krawtchouk}
\Delta(e+z)^{N}\Phi_{\mathbf{x}}\left(\left(e-\frac{1-p}{p}z\right)(e+z)^{-1}\right)
=\sum_{\mathbf{n}\subset N}\binom{N}{\mathbf{n}}K_{\mathbf{n}}(\mathbf{x};p,N)\Phi_{\mathbf{n}}(z).
\end{equation}
\end{thm}
\begin{proof}
{\rm{(1)}}\,We evaluate the spherical Taylor expansion of (\ref{eq:master gene fnc 1.2}) with respect to $w$: 
\begin{align}
\Phi_{\mathbf{x}}(\partial_{w})\Delta(e-z)^{-\alpha}\!\!\int_{K}\!\!e^{(kw|(e-\frac{1}{c}z)(e-z)^{-1})}\,dk\bigg|_{w=0}\!\!
&=\Delta(e-z)^{-\alpha}\!\!\int_{K}\!\!\Phi_{\mathbf{x}}(\partial_{w})e^{(w|k(e-\frac{1}{c}z)(e-z)^{-1})}|_{w=0}\,dk \nonumber \\
&=\Delta(e-z)^{-\alpha}\!\!\int_{K}\!\!\Phi_{\mathbf{x}}\left(\!k\!\left(\!\!\left(e-\frac{1}{c}z\right)\!(e-z)^{-1}\!\right)\!\!\right)\,dk \nonumber \\
&=\Delta(e-z)^{-\alpha}\Phi_{\mathbf{x}}\left(\!\left(e-\frac{1}{c}z\right)(e-z)^{-1}\!\right). \nonumber 
\end{align}
On the other hand, by (\ref{eq:master gene fnc 1.1}),
\begin{align}
\Phi_{\mathbf{x}}(\partial_{w})\Delta(e-z)^{-\alpha}\!\!\int_{K}\!\!e^{(kw|(e-\frac{1}{c}z)(e-z)^{-1})}\,dk\bigg|_{w=0}\!\!
&=\sum_{\mathbf{n} \in \mathscr{P}}d_{\mathbf{n}}\frac{(\alpha)_{\mathbf{n}}}{\left(\frac{n}{r}\right)_{\mathbf{n}}}M_{\mathbf{n}}(\mathbf{x};\alpha ,c)\Phi_{\mathbf{n}}(z). \nonumber 
\end{align}
Therefore, we obtain the conclusion.

\noindent
{\rm{(2)}}\,
The result is proved by a similar argument as in {\rm{(1)}}. 
That is, by {\rm{(2)}} of Theorem\,\ref{thm:master generating fnc 1}, we have 
\begin{align}
\sum_{\mathbf{n} \in \mathscr{P}}d_{\mathbf{n}}\frac{1}{\left(\frac{n}{r}\right)_{\mathbf{n}}}C_{\mathbf{n}}(\mathbf{x};a)\Phi_{\mathbf{n}}(z)
&=\Phi_{\mathbf{x}}(\partial_{w})e^{\tr{(w+z)}}\int_{K}e^{-\frac{1}{a}(kw|z)}\,dk\bigg|_{w=0} \nonumber \\
&=e^{\tr{z}}\int_{K}\!\!\Phi_{\mathbf{x}}(\partial_{w})e^{(w|k(e-\frac{1}{a}z))}|_{w=0}\,dk \nonumber \\
&=e^{\tr{z}}\int_{K}\!\!\Phi_{\mathbf{x}}\left(\!k\!\left(e-\frac{1}{a}z\right)\!\right)\,dk \nonumber \\
&=e^{\tr{z}}\Phi_{\mathbf{x}}\left(e-\frac{1}{a}z\right). \nonumber 
\end{align}

\noindent
{\rm{(3)}}\,From the assumption $\mathbf{x} \subset N$ and  (\ref{eq:gen fnc of Meixner}), we have
\begin{align}
\Delta(e-z)^{N}\Phi_{\mathbf{x}}\!\left(\!\left(e-\frac{1}{c}z\right)(e-z)^{-1}\!\right)
&=\lim_{\alpha \to -N}\Delta(e-z)^{-\alpha }\Phi_{\mathbf{x}}\left(\!\left(e-\frac{1}{c}z\right)(e-z)^{-1}\!\right) \nonumber \\
&=\sum_{\mathbf{n} \in \mathscr{P}}\frac{d_{\mathbf{n}}}{\left(\frac{n}{r}\right)_{\mathbf{n}}}
\sum_{\mathbf{k}\subset \mathbf{n}}\frac{\left(\frac{n}{r}\right)_{\mathbf{k}}}{d_{\mathbf{k}}}
\lim_{\alpha \to -N}\frac{(\alpha)_{\mathbf{n}}}{(\alpha)_{\mathbf{k}}}\binom{\mathbf{n}}{\mathbf{k}}\binom{\mathbf{x}}{\mathbf{k}}\left(1-\frac{1}{c}\right)^{|\mathbf{k}|} \nonumber \\
&=\sum_{\mathbf{n} \subset N}d_{\mathbf{n}}\frac{(-N)_{\mathbf{n}}}{\left(\frac{n}{r}\right)_{\mathbf{n}}}
\sum_{\mathbf{k}\subset \mathbf{x}}\frac{\left(\frac{n}{r}\right)_{\mathbf{k}}}{d_{\mathbf{k}}}
\frac{1}{(-N)_{\mathbf{k}}}\binom{\mathbf{n}}{\mathbf{k}}\binom{\mathbf{x}}{\mathbf{k}}\left(1-\frac{1}{c}\right)^{|\mathbf{k}|} \nonumber \\
&=\sum_{\mathbf{n} \subset N}\binom{N}{\mathbf{n}}M_{\mathbf{n}}(\mathbf{x};-N ,c)\Phi_{\mathbf{n}}(-z). \nonumber
\end{align}
Since this series is a finite sum, we can take $c=\frac{p}{p-1}$ above. Therefore, we obtain
\begin{align}
\Delta(e-z)^{N}\Phi_{\mathbf{x}}\left(\left(e+\frac{1-p}{p}z\right)(e-z)^{-1}\right)
&=\sum_{\mathbf{n} \subset N}\binom{N}{\mathbf{n}}M_{\mathbf{n}}\left(\mathbf{x};-N ,\frac{p}{p-1}\right)\Phi_{\mathbf{n}}(-z) \nonumber \\
&=\sum_{\mathbf{n} \subset N}\binom{N}{\mathbf{n}}K_{\mathbf{n}}\left(\mathbf{x};p,N\right)\Phi_{\mathbf{n}}(-z). \nonumber 
\end{align}
\end{proof}
\begin{rmk}
For $c=-1$, (\ref{eq:gen fnc of Meixner}) was obtained by Davidson-\'{O}lafsson-Zhang \cite{DOZ} (Lemma\,$4.1$) as a generating function of  multivariate Meixner-Pollaczeck polynomials which are called generalized Hermite polynomials in their paper.  
\end{rmk}
Next we apply the unitary transformations in (\ref{eq:unitary picture2}) to Theorem\,\ref{thm:master generating fnc 1}. 
Here, we also check convergence. 
\begin{lem}
\label{thm:absolute convergence of gen fnc of gen fnc for Meixner2}
Fix $0<c<1$ and let $0<\varepsilon  <1$ and $w,z \in \mathcal{D}$ satisfy that 
\begin{align}
\label{eq:absolute convergence condition of gen fnc of gen fnc for Meixner2}
& \left(c+(1-c)\frac{\varepsilon}{1-\varepsilon}\right)\left(1+(1-c)\frac{\varepsilon}{1-c\varepsilon}\right)<1, \nonumber \\
& |\Phi_{\mathbf{m}}(w)|, |\Phi_{\mathbf{m}}(z)|< \Phi_{\mathbf{m}}(\varepsilon)=\varepsilon^{|\mathbf{m}|}.  
\end{align}
Then, 
\begin{align}
\label{eq:absolute convergence of gen fnc of gen fnc for Meixner2}
&\sum_{\mathbf{x},\mathbf{m},\mathbf{n} \in \mathscr{P}}\left|d_{\mathbf{x}}\frac{(\alpha)_{\mathbf{x}}}{\left(\frac{n}{r}\right)_{\mathbf{x}}}c^{|\mathbf{x}|}
d_{\mathbf{m}}\frac{(\alpha)_{\mathbf{m}}}{\left(\frac{n}{r}\right)_{\mathbf{m}}}M_{\mathbf{m}}(\mathbf{x};\alpha ,c)\Phi_{\mathbf{m}}(z)
d_{\mathbf{n}}\frac{(\alpha)_{\mathbf{n}}}{\left(\frac{n}{r}\right)_{\mathbf{n}}}M_{\mathbf{n}}(\mathbf{x};\alpha ,c)\Phi_{\mathbf{n}}(cw)\right| \nonumber \\
{} & \quad<((1-c)(1-2(1+c)\varepsilon +(4c-1)\varepsilon^{2}))^{-r|\alpha|-dr(r-1)}.
\end{align}
\end{lem}
\begin{proof}
By Lemma\,\ref{thm:ineq for generalized shifted factorial} and Lemma\,\ref{thm:positivity of shifted Jack}, we have
\begin{align}
{\rm{(LHS)}}
&\leq  \sum_{\mathbf{x} \in \mathscr{P}}d_{\mathbf{x}}\frac{(|\alpha|+d(r-1))_{\mathbf{x}}}{\left(\frac{n}{r}\right)_{\mathbf{x}}}c^{|\mathbf{x}|} \nonumber \\
{} & \quad \cdot \sum_{\mathbf{k} \subset \mathbf{x}}\binom{\mathbf{x}}{\mathbf{k}}\frac{1}{(|\alpha|+d(r-1))_{\mathbf{k}}}\left(\frac{1}{c}-1\right)^{|\mathbf{k}|} 
\sum_{\mathbf{l} \subset \mathbf{x}}\binom{\mathbf{x}}{\mathbf{l}}\frac{1}{(|\alpha|+d(r-1))_{\mathbf{l}}}\left(\frac{1}{c}-1\right)^{|\mathbf{l}|} \nonumber \\
{} & \quad \cdot \!\!\sum_{\mathbf{m} \in \mathscr{P}}\!d_{\mathbf{m}}\frac{(|\alpha|+d(r-1))_{\mathbf{m}}}{\left(\frac{n}{r}\right)_{\mathbf{m}}}\gamma_{\mathbf{k}}(\mathbf{m}-\rho)\Phi_{\mathbf{m}}(\varepsilon)
\!\sum_{\mathbf{n} \in \mathscr{P}}d_{\mathbf{n}}\frac{(|\alpha|+d(r-1))_{\mathbf{n}}}{\left(\frac{n}{r}\right)_{\mathbf{n}}}\gamma_{\mathbf{l}}(\mathbf{n}-\rho)\Phi_{\mathbf{n}}(c\varepsilon). \nonumber 
\end{align}
Furthermore, from Lemma\,\ref{thm:spherical Taylor expan lem} and the definition of the generalized binomial coefficients (\ref{eq:the definition of the generalized binomial coefficients}), we derive
\begin{align}
{\rm{(LHS)}}
&\leq  ((1-\varepsilon)(1-c\varepsilon))^{-r|\alpha|-dr(r-1)}\sum_{\mathbf{x} \in \mathscr{P}}d_{\mathbf{x}}\frac{(|\alpha|+d(r-1))_{\mathbf{x}}}{\left(\frac{n}{r}\right)_{\mathbf{x}}}c^{|\mathbf{x}|} \nonumber \\
{} & \quad \cdot \sum_{\mathbf{k} \subset \mathbf{x}}\binom{\mathbf{x}}{\mathbf{k}}\left(\frac{1}{c}-1\right)^{|\mathbf{k}|}\Phi_{\mathbf{k}}\left(\frac{\varepsilon}{1-\varepsilon}\right)  
\sum_{\mathbf{l} \subset \mathbf{x}}\binom{\mathbf{x}}{\mathbf{l}}\left(\frac{1}{c}-1\right)^{|\mathbf{l}|}\Phi_{\mathbf{l}}\left(\frac{c\varepsilon}{1-c\varepsilon}\right) \nonumber \\
&=((1-\varepsilon)(1-c\varepsilon))^{-r|\alpha|-dr(r-1)} \nonumber \\
{} & \quad \cdot \sum_{\mathbf{x} \in \mathscr{P}}d_{\mathbf{x}}\frac{(|\alpha|+d(r-1))_{\mathbf{x}}}{\left(\frac{n}{r}\right)_{\mathbf{x}}}
\Phi_{\mathbf{x}}\left(\left(c+(1-c)\frac{\varepsilon}{1-\varepsilon}\right)\left(1+(1-c)\frac{\varepsilon}{1-c\varepsilon}\right)\right). \nonumber 
\end{align}
Finally, by using the assumption and Lemma\,\ref{thm:spherical Taylor expan lem}, we obtain
\begin{align}
{\rm{(LHS)}}
&\leq
\left((1-\varepsilon)(1-c\varepsilon)
\left(1-\left(c+(1-c)\frac{\varepsilon}{1-\varepsilon}\right)\left(1+(1-c)\frac{\varepsilon}{1-c\varepsilon}\right)\right)\right)^{-r|\alpha|-dr(r-1)} \nonumber \\
&=((1-c)(1-2(1+c)\varepsilon +(4c-1)\varepsilon^{2}))^{-r|\alpha|-dr(r-1)}. \nonumber
\end{align}
\end{proof}
From this lemma, we can consider the following generating functions.
\begin{thm}
\label{thm:master generating fnc 2}
{\rm{(1)}}\,For $z \in \mathcal{D}, u \in V^{\mathbb{C}}$, we obtain
\begin{align}
\label{eq:master generating fnc 2-1}
\sum_{\mathbf{m} \in \mathscr{P}}\psi_{\mathbf{m}}^{(\alpha)}(u)\Phi_{\mathbf{m}}(z)
&=\sum_{\mathbf{x},\mathbf{m} \in \mathscr{P}}d_{\mathbf{m}}\frac{(\alpha)_{\mathbf{m}}}{\left(\frac{n}{r}\right)_{\mathbf{m}}}M_{\mathbf{m}}(\mathbf{x};\alpha ,c)\Phi_{\mathbf{m}}(z) \nonumber \\
{} & \quad \cdot d_{\mathbf{x}}\frac{1}{\left(\frac{n}{r}\right)_{\mathbf{x}}}\left(\frac{2c}{1-c}\right)^{|\mathbf{x}|}e^{-\frac{1+c}{1-c}\tr{u}}\Phi_{\mathbf{x}}(u) \\
&=\Delta(e-z)^{-\alpha}\int_{K}e^{-(ku|(e+z)(e-z)^{-1})}\,dk. \nonumber
\end{align}
{\rm{(2)}}\,Fix $0<c<1$ and assume that $w,z \in \mathcal{D}$ satisfy the condition in {\rm{(1)}} of Lemma\,\ref{thm:absolute convergence of gen fnc of gen fnc for Meixner2}. 
We obtain
\begin{align}
\label{eq:master generating fnc 2-2}
\sum_{\mathbf{m} \in \mathscr{P}}d_{\mathbf{m}}\frac{(\alpha)_{\mathbf{m}}}{\left(\frac{n}{r}\right)_{\mathbf{m}}}\Phi_{\mathbf{m}}(w)\Phi_{\mathbf{m}}(z)
&=(1-c)^{r\alpha}\sum_{\mathbf{x},\mathbf{m} \in \mathscr{P}}d_{\mathbf{m}}\frac{(\alpha)_{\mathbf{m}}}{\left(\frac{n}{r}\right)_{\mathbf{m}}}M_{\mathbf{m}}(\mathbf{x};\alpha ,c)\Phi_{\mathbf{m}}(z) \nonumber \\
{} & \quad \cdot d_{\mathbf{x}}\frac{(\alpha)_{\mathbf{x}}}{\left(\frac{n}{r}\right)_{\mathbf{x}}}c^{|\mathbf{x}|}\Delta(e-cw)^{-\alpha}\Phi_{\mathbf{x}}((e-w)(e-cw)^{-1}) \\
&=\Delta(z)^{-\alpha}\int_{K}\Delta(kz^{-1}-w)^{-\alpha}\,dk. \nonumber
\end{align}
{\rm{(3)}}\,For $w, z \in V^{\mathbb{C}}$, we obtain
\begin{align}
\label{eq:master generating fnc 2-5}
\sum_{\mathbf{m} \in \mathscr{P}}d_{\mathbf{m}}\frac{1}{\left(\frac{n}{r}\right)_{\mathbf{m}}}\Phi_{\mathbf{m}}(w)\Phi_{\mathbf{m}}(z)
&=e^{-ra}\sum_{\mathbf{x},\mathbf{m} \in \mathscr{P}}d_{\mathbf{m}}\frac{a^{|\mathbf{m}|}}{\left(\frac{n}{r}\right)_{\mathbf{m}}}C_{\mathbf{m}}(\mathbf{x};a)\Phi_{\mathbf{m}}(z) \nonumber \\
{} & \quad \cdot d_{\mathbf{x}}\frac{a^{|\mathbf{x}|}}{\left(\frac{n}{r}\right)_{\mathbf{x}}}e^{\tr{w}}\Phi_{\mathbf{x}}\left(e-\frac{1}{a}w\right) \\
&=e^{\tr{w}}\int_{K}e^{-a(kw|e-z)}\,dk. \nonumber
\end{align}
\end{thm}
\begin{proof}
As {\rm{(1)}} and {\rm{(3)}} follow immediately from Theorem\,\ref{thm:master generating fnc 1}, we only prove {\rm{(2)}}. 

First, we remark that the right hand side of (\ref{eq:master generating fnc 2-2}) converges absolutely under the conditions in Lemma\,\ref{thm:absolute convergence of gen fnc of gen fnc for Meixner2}.  
Moreover, we also remark that since (\ref{eq:absolute conv for MLP})
$$
\sum_{\mathbf{m} \in \mathscr{P}}|\psi_{\mathbf{m}}^{(\alpha)}(u)\Phi_{\mathbf{m}}(z)|\leq (1-a_{1})^{-r|\alpha|-dr(r-1)}e^{-\frac{1-3a_{1}}{1-a_{1}}}, 
$$
the exchange of unitary transformations $\mathcal{L}_{\alpha}$, $\mathcal{M}_{\alpha,\theta}$ and $\mathcal{F}_{\alpha,\nu}^{-1}$, and the summation are  justified under these restrictions. 
Therefore, to obtain the results, we apply the unitary transforms to both sides of (\ref{eq:master generating fnc 2-1}). 
We will perform these calculations.

For {\rm{(2)}}, 
we apply transform $C_{\alpha}^{-1}\circ \mathcal{L}_{\alpha}$ to both sides of (\ref{eq:master generating fnc 2-1}). 
From Lemma\,\ref{thm:int formula}, we have 
\begin{align}
\mathcal{L}_{\alpha}(e^{-\frac{1+c}{1-c}\tr{u}}\Phi_{\mathbf{x}})(z)
&=\frac{2^{r\alpha}}{\Gamma_{\Omega}(\alpha)}\int_{\Omega}e^{-\left(\frac{1+c}{1-c}e+z\big|u\right)}\Phi_{\mathbf{x}}(u)\Delta(u)^{\alpha -\frac{n}{r}}\,du \nonumber \\
&=2^{r\alpha}(\alpha)_{\mathbf{x}}\Delta\left(\frac{1+c}{1-c}e+z\right)^{-\alpha}\Phi_{\mathbf{x}}\left(\left(\frac{1+c}{1-c}e+z\right)^{-1}\right). \nonumber 
\end{align}
Furthermore, 
\begin{align}
C_{\alpha}^{-1}\circ \mathcal{L}_{\alpha}(e^{-\frac{1+c}{1-c}\tr{u}}\Phi_{\mathbf{x}})(w)
&=2^{r\alpha}(\alpha)_{\mathbf{x}}C_{\alpha}^{-1}\left(\Delta\left(\frac{1+c}{1-c}e+z\right)^{-\alpha}\Phi_{\mathbf{x}}\left(\left(\frac{1+c}{1-c}e+z\right)^{-1}\right)\right)(w) \nonumber \\
&=(\alpha)_{\mathbf{x}}\Delta(e-w)^{-\alpha}\Delta\left(\frac{1}{2}\left((e+w)(e-w)^{-1}+\frac{1+c}{1-c}e\right)\right)^{-\alpha} \nonumber \\
{} & \quad \cdot \Phi_{\mathbf{x}}\left(\left((e+w)(e-w)^{-1}+\frac{1+c}{1-c}e\right)^{-1}\right) \nonumber \\
&=(\alpha)_{\mathbf{x}}(1-c)^{r\alpha}\Delta(e-cw)^{-\alpha}\Phi_{\mathbf{x}}\left(\frac{1-c}{2}\,(e-w)(e-cw)^{-1}\right). \nonumber
\end{align}
Hence, the right-hand side of (\ref{eq:master generating fnc 2-1}) becomes the right-hand side of (\ref{eq:master generating fnc 2-2}).
Therefore, since $C_{\alpha}^{-1}\circ \mathcal{L}_{\alpha}(\psi_{\mathbf{m}}^{(\alpha)})(w)=d_{\mathbf{m}}\frac{(\alpha)_{\mathbf{m}}}{\left(\frac{n}{r}\right)_{\mathbf{m}}}\Phi_{\mathbf{m}}(w)$, we obtain the conclusion.
\end{proof}

\subsection{Orthogonality relations}
We provide the orthogonality relations for our discrete orthogonal polynomials as a corollary of Theorem\,\ref{thm:master generating fnc 2}.
\begin{thm}
\label{thm:Orthogonality of MDOP}
{\rm{(1)}}\,
For $\alpha >\frac{n}{r}-1, 0<c<1$, we obtain
\begin{align}
\label{eq:Orthogonality of Meixner}
\sum_{\mathbf{x} \in \mathscr{P}}d_{\mathbf{x}}\frac{(\alpha)_{\mathbf{x}}}{\left(\frac{n}{r}\right)_{\mathbf{x}}}c^{|\mathbf{x}|}M_{\mathbf{m}}(\mathbf{x};\alpha ,c)M_{\mathbf{n}}(\mathbf{x};\alpha ,c)
=\frac{c^{-|\mathbf{m}|}}{(1-c)^{r\alpha}}\frac{1}{d_{\mathbf{m}}}\frac{\left(\frac{n}{r}\right)_{\mathbf{m}}}{(\alpha)_{\mathbf{m}}}\delta_{\mathbf{m},\mathbf{n}}\geq 0.
\end{align}
{\rm{(2)}}\,
For $a >0$, we obtain
\begin{align}
\label{eq:Orthogonality of Charlier}
\sum_{\mathbf{x} \in \mathscr{P}}d_{\mathbf{x}}\frac{a^{|\mathbf{x}|}}{\left(\frac{n}{r}\right)_{\mathbf{x}}}C_{\mathbf{m}}(\mathbf{x};a)C_{\mathbf{n}}(\mathbf{x};a)
=a^{-|\mathbf{m}|}e^{ra}\frac{\left(\frac{n}{r}\right)_{\mathbf{m}}}{d_{\mathbf{m}}}\delta_{\mathbf{m},\mathbf{n}}\geq 0.
\end{align}
{\rm{(3)}}\,
For $0<p<1$, we obtain
\begin{equation}
\label{eq;Orthogonality of Krawtchouk}
\sum_{\mathbf{x}\subset N}\binom{N}{\mathbf{x}}p^{|\mathbf{x}|}(1-p)^{rN-|\mathbf{x}|}K_{\mathbf{m}}(\mathbf{x};p,N)K_{\mathbf{n}}(\mathbf{x};p,N)
=\left(\frac{1-p}{p}\right)^{|\mathbf{m}|}\binom{N}{\mathbf{m}}^{-1}\delta_{\mathbf{m},\mathbf{n}}\geq 0.
\end{equation}
\end{thm}
\begin{proof}
{\rm{(1)}}\,From (\ref{eq:master generating fnc 2-2}) and (\ref{eq:gen fnc of Meixner}), we have
\begin{align}
\sum_{\mathbf{m} \in \mathscr{P}}d_{\mathbf{m}}\frac{(\alpha)_{\mathbf{m}}}{\left(\frac{n}{r}\right)_{\mathbf{m}}}\Phi_{\mathbf{m}}(w)\Phi_{\mathbf{m}}(z)
&=(1-c)^{r\alpha}\sum_{\mathbf{x},\mathbf{m} \in \mathscr{P}}d_{\mathbf{m}}\frac{(\alpha)_{\mathbf{m}}}{\left(\frac{n}{r}\right)_{\mathbf{m}}}M_{\mathbf{m}}(\mathbf{x};\alpha ,c)\Phi_{\mathbf{m}}(z) \nonumber \\
{} & \quad \cdot d_{\mathbf{x}}\frac{(\alpha)_{\mathbf{x}}}{\left(\frac{n}{r}\right)_{\mathbf{x}}}c^{|\mathbf{x}|}\Delta(e-cw)^{-\alpha}\Phi_{\mathbf{x}}((e-w)(e-cw)^{-1}) \nonumber \\
&=\sum_{\mathbf{m},\mathbf{n} \in \mathscr{P}}(1-c)^{r\alpha}
d_{\mathbf{m}}\frac{(\alpha)_{\mathbf{m}}}{\left(\frac{n}{r}\right)_{\mathbf{m}}} 
d_{\mathbf{n}}\frac{(\alpha)_{\mathbf{n}}}{\left(\frac{n}{r}\right)_{\mathbf{n}}}c^{|\mathbf{n}|} \nonumber \\
{} & \quad \cdot 
\left\{\sum_{\mathbf{x} \in \mathscr{P}}d_{\mathbf{x}}\frac{(\alpha)_{\mathbf{x}}}{\left(\frac{n}{r}\right)_{\mathbf{x}}}c^{|\mathbf{x}|}M_{\mathbf{m}}(\mathbf{x};\alpha ,c)M_{\mathbf{n}}(\mathbf{x};\alpha ,c)\right\}
\Phi_{\mathbf{m}}(z)\Phi_{\mathbf{n}}(w). \nonumber
\end{align}
Therefore, by comparing the coefficients of $\Phi_{\mathbf{m}}(z)\Phi_{\mathbf{n}}(w)$ on both sides of this equation, we obtain (\ref{eq:Orthogonality of Meixner}).

\noindent
{\rm{(2)}}\,From (\ref{eq:master generating fnc 2-5}) and (\ref{eq:gen fnc of Charlier}), we derive
\begin{align}
\sum_{\mathbf{m} \in \mathscr{P}}d_{\mathbf{m}}\frac{1}{\left(\frac{n}{r}\right)_{\mathbf{m}}}\Phi_{\mathbf{m}}(w)\Phi_{\mathbf{m}}(z)
&=e^{-ra}\sum_{\mathbf{x},\mathbf{m} \in \mathscr{P}}d_{\mathbf{m}}\frac{a^{|\mathbf{m}|}}{\left(\frac{n}{r}\right)_{\mathbf{m}}}C_{\mathbf{m}}(\mathbf{x};a)\Phi_{\mathbf{m}}(z) \nonumber \\
{} & \quad \cdot d_{\mathbf{x}}\frac{a^{|\mathbf{x}|}}{\left(\frac{n}{r}\right)_{\mathbf{x}}}e^{\tr{w}}\Phi_{\mathbf{x}}\left(e-\frac{1}{a}w\right) \nonumber \\
&=\sum_{\mathbf{m},\mathbf{n} \in \mathscr{P}}e^{-ra}
d_{\mathbf{m}}\frac{a^{|\mathbf{m}|}}{\left(\frac{n}{r}\right)_{\mathbf{m}}}
d_{\mathbf{n}}\frac{1}{\left(\frac{n}{r}\right)_{\mathbf{n}}} \nonumber \\
{} & \quad \cdot \left\{\sum_{\mathbf{x} \in \mathscr{P}}d_{\mathbf{x}}\frac{a^{|\mathbf{x}|}}{\left(\frac{n}{r}\right)_{\mathbf{x}}}C_{\mathbf{m}}(\mathbf{x};a)C_{\mathbf{n}}(\mathbf{x};a)\right\}\Phi_{\mathbf{m}}(z)\Phi_{\mathbf{n}}(w). \nonumber
\end{align}
Then, by comparing the coefficients of $\Phi_{\mathbf{m}}(z)\Phi_{\mathbf{n}}(w)$, we have the conclusion.

\noindent
{\rm{(3)}}\,
In (\ref{eq:master generating fnc 2-2}), taking $\alpha =-N$, one has
\begin{align}
\sum_{\mathbf{m} \in \mathscr{P}}d_{\mathbf{m}}\frac{(-N)_{\mathbf{m}}}{\left(\frac{n}{r}\right)_{\mathbf{m}}}\Phi_{\mathbf{m}}(w)\Phi_{\mathbf{m}}(-z)
&=\sum_{\mathbf{m} \subset N}\binom{N}{\mathbf{m}}\Phi_{\mathbf{m}}(w)\Phi_{\mathbf{m}}(z) \nonumber \\
&=(1-c)^{-rN}\sum_{\mathbf{x},\mathbf{m} \subset N}d_{\mathbf{m}}\frac{(-N)_{\mathbf{m}}}{\left(\frac{n}{r}\right)_{\mathbf{m}}}M_{\mathbf{m}}(\mathbf{x};-N ,c)\Phi_{\mathbf{m}}(-z) \nonumber \\
{} & \quad \cdot d_{\mathbf{x}}\frac{(-N)_{\mathbf{x}}}{\left(\frac{n}{r}\right)_{\mathbf{x}}}c^{|\mathbf{x}|}\Delta(e-cw)^{N}\Phi_{\mathbf{x}}((e-w)(e-cw)^{-1}). \nonumber
\end{align}
The first equality follows from (\ref{eq:gamma special case}). 
Since the above sum is finite, we can put $c=\frac{p}{p-1},\,\,(0<p<1)$. 
Hence,
\begin{align}
\sum_{\mathbf{m} \subset N}\binom{N}{\mathbf{m}}\Phi_{\mathbf{m}}(w)\Phi_{\mathbf{m}}(z)
&=(1-p)^{rN}\sum_{\mathbf{x},\mathbf{m} \subset N}\binom{N}{\mathbf{m}}K_{\mathbf{m}}(\mathbf{x};p ,N)\Phi_{\mathbf{m}}(z)\binom{N}{\mathbf{x}}\left(\frac{p}{1-p}\right)^{|\mathbf{x}|} \nonumber \\
{} & \quad \cdot 
\Delta\left(e+\frac{p}{1-p}w\right)^{N}\Phi_{\mathbf{x}}\left((e-w)\left(e+\frac{p}{1-p}w\right)^{-1}\right). \nonumber
\end{align}
From (\ref{eq:gen fnc of Krawtchouk}), we have 
$$
\Delta\left(e+\frac{p}{1-p}w\right)^{N}\!\Phi_{\mathbf{x}}\left(\!(e-w)\!\left(e+\frac{p}{1-p}w\right)^{-1}\right)
\!=\!\sum_{\mathbf{n}\subset N}\binom{N}{\mathbf{n}}K_{\mathbf{n}}(\mathbf{x};p,N)\!\left(\frac{p}{1-p}\right)^{|\mathbf{n}|}\!\Phi_{\mathbf{n}}(w). 
$$
Therefore, 
\begin{align}
\sum_{\mathbf{m} \subset N}\binom{N}{\mathbf{m}}\Phi_{\mathbf{m}}(w)\Phi_{\mathbf{m}}(z)
&=\sum_{\mathbf{m},\mathbf{n} \subset N}\binom{N}{\mathbf{m}}
\left(\frac{p}{1-p}\right)^{|\mathbf{n}|}\binom{N}{\mathbf{n}} \nonumber \\ 
{} & \quad \!\! \cdot \!\!
\left\{\sum_{\mathbf{x} \subset N}\!\binom{N}{\mathbf{x}}p^{|\mathbf{x}|}(1-p)^{rN-|\mathbf{x}|}K_{\mathbf{m}}(\mathbf{x};p ,N)K_{\mathbf{n}}(\mathbf{x};p ,N)\!\right\}\!\Phi_{\mathbf{m}}(z)\Phi_{\mathbf{n}}(w). \nonumber
\end{align}
\end{proof}

\subsection{Difference equations and recurrence relations}
In this subsection, we derive the difference equations and recurrence formulas for our polynomials 
from (\ref{eq:differential equation for Laguerre}), Lemma\,\ref{thm:differential Lemma for Laguerre and difference Lemma for Meixner} and {\rm{(1)}} of Theorem\,\ref{thm:master generating fnc 2}.
\begin{thm}
\label{thm:Difference eq of MDOP}
{\rm{(1)}}\,
For $\mathbf{x},\mathbf{m} \in \mathscr{P}$, we have
\begin{align}
d_{\mathbf{x}}(c-1)|\mathbf{m}|M_{\mathbf{m}}(\mathbf{x};\alpha ,c)
&=\sum_{j=1}^{r}d_{\mathbf{x}+\epsilon_{j}}\tilde{a}_{j}(-\mathbf{x}-\epsilon_{j})\left(x_{j}+\alpha-\frac{d}{2}(j-1)\right){c}M_{\mathbf{m}}(\mathbf{x}+\epsilon_{j};\alpha,c) \nonumber \\
{} & \quad -\sum_{j=1}^{r}d_{\mathbf{x}}(x_{j}+(x_{j}+\alpha){c})M_{\mathbf{m}}(\mathbf{x};\alpha,c) \nonumber \\
\label{eq:Difference eq of Meixner}
{} & \quad +\sum_{j=1}^{r}d_{\mathbf{x}-\epsilon_{j}}\tilde{a}_{j}(\mathbf{x}-\epsilon_{j})\left(x_{j}+\frac{d}{2}(r-j)\right)M_{\mathbf{m}}(\mathbf{x}-\epsilon_{j};\alpha,c).
\end{align}
{\rm{(2)}}\,
For $\mathbf{x},\mathbf{m} \in \mathscr{P}$, we have
\begin{align}
-d_{\mathbf{x}}|\mathbf{m}|C_{\mathbf{m}}(\mathbf{x};a)
&=\sum_{j=1}^{r}d_{\mathbf{x}+\epsilon_{j}}\tilde{a}_{j}(-\mathbf{x}-\epsilon_{j}){a}C_{\mathbf{m}}(\mathbf{x}+\epsilon_{j};a) \nonumber \\
{} & \quad -\sum_{j=1}^{r}d_{\mathbf{x}}(x_{j}+a)C_{\mathbf{m}}(\mathbf{x};a) \nonumber \\
\label{eq:Difference eq of Charlier}
{} & \quad +\sum_{j=1}^{r}d_{\mathbf{x}-\epsilon_{j}}\tilde{a}_{j}(\mathbf{x}-\epsilon_{j})\left(x_{j}+\frac{d}{2}(r-j)\right)C_{\mathbf{m}}(\mathbf{x}-\epsilon_{j};a).
\end{align}
{\rm{(3)}}\,
For $\mathbf{x},\mathbf{m} \subset N$, we have
\begin{align}
-d_{\mathbf{x}}|\mathbf{m}|K_{\mathbf{m}}(\mathbf{x};p ,N)
&=\sum_{j=1}^{r}d_{\mathbf{x}+\epsilon_{j}}\tilde{a}_{j}(-\mathbf{x}-\epsilon_{j})\left(N-x_{j}+\frac{d}{2}(j-1)\right){p}K_{\mathbf{m}}(\mathbf{x}+\epsilon_{j};p,N) \nonumber \\
{} & \quad -\sum_{j=1}^{r}d_{\mathbf{x}}(p(N-x_{j})+x_{j}(1-p))K_{\mathbf{m}}(\mathbf{x};p,N) \nonumber \\
\label{eq:Difference eq of Krawtchouk}
{} & \quad +\sum_{j=1}^{r}d_{\mathbf{x}-\epsilon_{j}}\tilde{a}_{j}(\mathbf{x}-\epsilon_{j})\left(x_{j}+\frac{d}{2}(r-j)\right)(1-p)K_{\mathbf{m}}(\mathbf{x}-\epsilon_{j};p,N).
\end{align}
\end{thm}
\begin{proof}
{\rm{(1)}}\,Let us apply operator $\frac{c-1}{2}e^{\frac{1+c}{1-c}\tr{u}}D_{\alpha}^{(1)}$ to both sides of (\ref{eq:master generating fnc 2-1}). 
Since $D_{\alpha}^{(1)}\psi_{\mathbf{m}}^{(\alpha)}(u)=2|\mathbf{m}|\psi_{\mathbf{m}}^{(\alpha)}(u)$, we have
\begin{align}
\frac{c-1}{2}e^{\frac{1+c}{1-c}\tr{u}}D_{\alpha}^{(1)}\left(\sum_{\mathbf{m} \in \mathscr{P}}\psi_{\mathbf{m}}^{(\alpha)}(u)\Phi_{\mathbf{m}}(z)\right)
&=\sum_{\mathbf{m} \in \mathscr{P}}(c-1)e^{\frac{1+c}{1-c}\tr{u}}|\mathbf{m}|\psi_{\mathbf{m}}^{(\alpha)}(u)\Phi_{\mathbf{m}}(z) \nonumber \\
&=\sum_{\mathbf{x},\mathbf{m} \in \mathscr{P}}d_{\mathbf{m}}\frac{(\alpha)_{\mathbf{m}}}{\left(\frac{n}{r}\right)_{\mathbf{m}}}\left(\frac{2c}{1-c}\right)^{|\mathbf{x}|}\Phi_{\mathbf{x}}(u)\Phi_{\mathbf{m}}(z) \nonumber \\
{} & \quad \cdot d_{\mathbf{x}}\frac{1}{\left(\frac{n}{r}\right)_{\mathbf{x}}}(c-1)|\mathbf{m}|M_{\mathbf{m}}(\mathbf{x};\alpha ,c). \nonumber
\end{align}
On the other hand, by (\ref{eq:conj differntial of D3 and spherical poly}), we have
\begin{align}
&\frac{c-1}{2}e^{\frac{1+c}{1-c}\tr{u}}D_{\alpha}^{(1)}\left(\sum_{\mathbf{m} \in \mathscr{P}}\psi_{\mathbf{m}}^{(\alpha)}(u)\Phi_{\mathbf{m}}(z)\right) \nonumber \\
{} & \quad =\sum_{\mathbf{x},\mathbf{m} \in \mathscr{P}}d_{\mathbf{m}}\frac{(\alpha)_{\mathbf{m}}}{\left(\frac{n}{r}\right)_{\mathbf{m}}}M_{\mathbf{m}}(\mathbf{x};\alpha ,c)\Phi_{\mathbf{m}}(z)d_{\mathbf{x}}\frac{1}{\left(\frac{n}{r}\right)_{\mathbf{x}}}\left(\frac{2c}{1-c}\right)^{|\mathbf{x}|}
\frac{c-1}{2}e^{\frac{1+c}{1-c}\tr{u}}D_{\alpha}^{(1)}(e^{-\frac{1+c}{1-c}\tr{u}}\Phi_{\mathbf{x}}(u)) \nonumber \\
{} & \quad =\sum_{\mathbf{x},\mathbf{m} \in \mathscr{P}}d_{\mathbf{m}}\frac{(\alpha)_{\mathbf{m}}}{\left(\frac{n}{r}\right)_{\mathbf{m}}}M_{\mathbf{m}}(\mathbf{x};\alpha ,c)\Phi_{\mathbf{m}}(z)d_{\mathbf{x}}\frac{1}{\left(\frac{n}{r}\right)_{\mathbf{x}}}\left(\frac{2c}{1-c}\right)^{|\mathbf{x}|} \nonumber \\
{} & \quad \quad \cdot \left\{\frac{2c}{1-c}\sum_{j=1}^{r}\tilde{a_{j}}(\mathbf{x})\Phi_{\mathbf{x}+\epsilon_{j}}(u) 
-\sum_{j=1}^{r}(x_{j}+(x_{j}+\alpha)c)\Phi_{\mathbf{x}}(u) \right. \nonumber \\
{} & \quad \quad \quad \left. +\frac{1-c}{2}\sum_{j=1}^{r}\left(x_{j}+\frac{d}{2}(r-j)\right)\left(x_{j}+\alpha- 1-\frac{d}{2}(j-1)\right)\tilde{a_{j}}(-\mathbf{x})\Phi_{\mathbf{x}-\epsilon_{j}}(u)\right\} \nonumber \\
{} & \quad =\sum_{\mathbf{x},\mathbf{m} \in \mathscr{P}}d_{\mathbf{m}}\frac{(\alpha)_{\mathbf{m}}}{\left(\frac{n}{r}\right)_{\mathbf{m}}}\left(\frac{2c}{1-c}\right)^{|\mathbf{x}|}\Phi_{\mathbf{x}}(u)\Phi_{\mathbf{m}}(z) \nonumber \\
{} & \quad \quad \cdot 
\left\{\sum_{j=1}^{r}d_{\mathbf{x}+\epsilon_{j}}\tilde{a_{j}}(-\mathbf{x}-\epsilon_{j})\frac{x_{j}+1+\frac{d}{2}(r-j)}{\left(\frac{n}{r}\right)_{\mathbf{x}+\epsilon_{j}}}\left(x_{j}+\alpha-\frac{d}{2}(j-1)\right)c
M_{\mathbf{m}}(\mathbf{x}+\epsilon_{j};\alpha ,c) \right. \nonumber \\
{} & \quad \quad \quad \left. -\sum_{j=1}^{r}d_{\mathbf{x}}\frac{1}{\left(\frac{n}{r}\right)_{\mathbf{x}}}(x_{j}+(x_{j}+\alpha)c)M_{\mathbf{m}}(\mathbf{x};\alpha ,c) \right. \nonumber \\
{} & \quad \quad \quad \left. +\sum_{j=1}^{r}d_{\mathbf{x}-\epsilon_{j}}\tilde{a_{j}}(\mathbf{x}-\epsilon_{j})\frac{1}{\left(\frac{n}{r}\right)_{\mathbf{x}-\epsilon_{j}}}M_{\mathbf{m}}(\mathbf{x}-\epsilon_{j};\alpha ,c)\right\}. \nonumber
\end{align}
Finally, the conclusion is obtained by 
$$
\left(\frac{n}{r}\right)_{\mathbf{x}+\epsilon_{j}}=\left(x_{j}+1+\frac{d}{2}(r-j)\right)\left(\frac{n}{r}\right)_{\mathbf{x}}
$$
and comparing the coefficients in the above. 

\noindent
{\rm{(2)}}\,
Put $c=\frac{a}{a+\alpha}$ in (\ref{eq:Difference eq of Meixner}) and take the limit as $\alpha \to \infty$. 
Then, by (\ref{eq:relation Meixner and Charlier}), we have the conclusion. 

\noindent
{\rm{(3)}}\,
Put $c=\frac{p}{p-1}, \alpha =-N$ and multiply $1-p$ in (\ref{eq:Difference eq of Meixner}). 
Then, by (\ref{eq:relation Meixner and Krawtchouk}), we have the conclusion.
\end{proof}

The recurrence formulas follow immediately from Theorem\,\ref{thm:Difference eq of MDOP} and Proposition\,\ref{thm:duality of MDOP}.
\begin{thm}
\label{thm:Recurrence formula of MDOP}
{\rm{(1)}}\,
For $\mathbf{x},\mathbf{m} \in \mathscr{P}$, we have
\begin{align}
d_{\mathbf{m}}(c-1)|\mathbf{x}|M_{\mathbf{m}}(\mathbf{x};\alpha ,c)
&=\sum_{j=1}^{r}d_{\mathbf{m}+\epsilon_{j}}\tilde{a}_{j}(-\mathbf{m}-\epsilon_{j})\left(m_{j}+\alpha-\frac{d}{2}(j-1)\right){c}M_{\mathbf{m}+\epsilon_{j}}(\mathbf{x};\alpha,c) \nonumber \\
{} & \quad -\sum_{j=1}^{r}d_{\mathbf{m}}(m_{j}+(m_{j}+\alpha){c})M_{\mathbf{m}}(\mathbf{x};\alpha,c) \nonumber \\
{} & \quad +\sum_{j=1}^{r}d_{\mathbf{m}-\epsilon_{j}}\tilde{a}_{j}(\mathbf{m}-\epsilon_{j})\left(m_{j}+\frac{d}{2}(r-j)\right)M_{\mathbf{m}-\epsilon_{j}}(\mathbf{x};\alpha,c).
\end{align}
{\rm{(2)}}\,
For $\mathbf{x},\mathbf{m} \in \mathscr{P}$, we have
\begin{align}
-d_{\mathbf{m}}|\mathbf{x}|C_{\mathbf{m}}(\mathbf{x};a)
&=\sum_{j=1}^{r}d_{\mathbf{m}+\epsilon_{j}}\tilde{a}_{j}(-\mathbf{m}-\epsilon_{j}){a}C_{\mathbf{m}+\epsilon_{j}}(\mathbf{x};a) \nonumber \\
{} & \quad -\sum_{j=1}^{r}d_{\mathbf{m}}(m_{j}+a)C_{\mathbf{m}}(\mathbf{x};a) \nonumber \\
{} & \quad +\sum_{j=1}^{r}d_{\mathbf{m}-\epsilon_{j}}\tilde{a}_{j}(\mathbf{m}-\epsilon_{j})\left(m_{j}+\frac{d}{2}(r-j)\right)C_{\mathbf{m}-\epsilon_{j}}(\mathbf{x};a).
\end{align}
{\rm{(3)}}\,
For $\mathbf{x},\mathbf{m} \subset N$, we have
\begin{align}
-d_{\mathbf{m}}|\mathbf{x}|K_{\mathbf{m}}(\mathbf{x};p ,N)
&=\sum_{j=1}^{r}d_{\mathbf{m}+\epsilon_{j}}\tilde{a}_{j}(-\mathbf{m}-\epsilon_{j})\left(N-m_{j}+\frac{d}{2}(j-1)\right){p}K_{\mathbf{m}+\epsilon_{j}}(\mathbf{x};p,N) \nonumber \\
{} & \quad -\sum_{j=1}^{r}d_{\mathbf{m}}(p(N-m_{j})+m_{j}(1-p))K_{\mathbf{m}}(\mathbf{x};p,N) \nonumber \\
{} & \quad +\sum_{j=1}^{r}d_{\mathbf{m}-\epsilon_{j}}\tilde{a}_{j}(\mathbf{m}-\epsilon_{j})\left(m_{j}+\frac{d}{2}(r-j)\right)(1-p)K_{\mathbf{m}-\epsilon_{j}}(\mathbf{x};p,N).
\end{align}
\end{thm}

\subsection{Determinant formulas}
In this subsection, we assume $d=2$ (In particular, we remark that $\frac{n}{r}=r$). 
In this case the spherical polynomials $\Phi_{\mathbf{m}}$ are proportional to the Schur polynomials $s_{\mathbf{m}}$ (recall (\ref{eq:spherical and schur})). 
Further, there are some determinant formulas for the multivariate Meixner, Charlier and Krawtchouk polynomials. 
Before stating the main theorem, we provide the following lemma needed to prove the determinant formulas. 
\begin{lem}[$\cite{H}$\,Theorem\,1.2.1]
\label{thm:det lem}
Consider $r$ power series of single variable $z \in \mathbb{C}$
$$
f_{\mu}(z)=\sum_{m\geq 0}A_{m}^{(\mu)}z^{m}\,\,\,\,(\mu=1,\cdots,r).
$$
Then, 
\begin{equation}
\label{eq:det lem}
\frac{\det{(f_{\mu}(z_{\nu}))}}{V(z_{1},\ldots,z_{r})}=\sum_{\mathbf{m} \in \mathscr{P}}A_{\mathbf{m}}s_{\mathbf{m}}(z_{1},\ldots,z_{r}), 
\end{equation}
where $V(z_{1},\ldots,z_{r})$ denote the Vandermonde determinant
$$
V(z_{1},\ldots,z_{r}):=\prod_{1\leq \mu <\nu \leq r}(z_{\mu}-z_{\nu}), 
$$
and 
$$
A_{\mathbf{m}}:=\det{(A_{m_{\mu}+r-\mu}^{(\nu)})}. 
$$
\end{lem}

\begin{thm}
{\rm{(1)}}\,
For any $\mathbf{m},\mathbf{x} \in \mathscr{P}$, we obtain
\begin{align}
M_{\mathbf{m}}(\mathbf{x};\alpha ,c)
&=\frac{1}{\delta !}\frac{\left(1-\frac{1}{c}\right)^{-\frac{r(r-1)}{2}}}{s_{\mathbf{m}}(1,\ldots,1)s_{\mathbf{x}}(1,\ldots,1)}\prod_{j=1}^{r}(\alpha -r+1)_{j-1} \nonumber \\
\label{eq:det Meixner}
{} & \quad \cdot \det(M_{m_{\mu}+r-\mu}(x_{\nu}+r-\nu;\alpha -r+1,c)). 
\end{align}
Here, $M_{m_{\mu}+r-\mu}(x_{\nu}+r-\nu;\alpha -r+1,c)$ is a one variable Meixner polynomial.

\noindent
{\rm{(2)}}\,
For any $\mathbf{m},\mathbf{x} \in \mathscr{P}$, we obtain
\begin{align}
\label{eq:det Charlier}
C_{\mathbf{m}}(\mathbf{x};a)
=\frac{1}{\delta !}\frac{(-a)^{\frac{r(r-1)}{2}}}{s_{\mathbf{m}}(1,\ldots,1)s_{\mathbf{x}}(1,\ldots,1)}
\det(C_{m_{\mu}+r-\mu}(x_{\nu}+r-\nu;a)).
\end{align}
Here, $C_{m_{\mu}+r-\mu}(x_{\nu}+r-\nu;a)$ is a one variable Charlier polynomial.

\noindent
{\rm{(3)}}\,
For any $\mathbf{m},\mathbf{x} \subset N=(N,\ldots,N)$, we obtain
\begin{align}
K_{\mathbf{m}}(\mathbf{x};p,N)
&=\frac{1}{\delta !}\frac{p^{\frac{r(r-1)}{2}}}{s_{\mathbf{m}}(1,\ldots,1)s_{\mathbf{x}}(1,\ldots,1)}\prod_{j=1}^{r}(-N-r+1)_{j-1} \nonumber \\
\label{eq:det Krawtchouk}
{} & \quad \cdot \det(K_{m_{\mu}+r-\mu}(x_{\nu}+r-\nu;p,N+r-1)).
\end{align}
Here, $K_{m_{\mu}+r-\mu}(x_{\nu}+r-\nu;p,N+r-1)$ is a one variable Krawtchouk polynomial.
\end{thm}
\begin{proof}
{\rm{(1)}}\,Let put $z=\sum_{j=1}^{r}z_{j}c_{j}, (0<z_{1},\ldots,z_{r}<1)$. 
Since 
$$
s_{\mathbf{m}}(z_{1},\ldots,z_{r})=s_{\mathbf{m}}(1,\ldots,1)\Phi_{\mathbf{m}}(z), 
$$ 
the generating function of the multivariate Meixner polynomials (\ref{eq:gen fnc of Meixner}) express
\begin{align}
\sum_{\mathbf{m} \in \mathscr{P}}d_{\mathbf{m}}\frac{(\alpha)_{\mathbf{m}}}{\left(r\right)_{\mathbf{m}}}M_{\mathbf{m}}(\mathbf{x};\alpha ,c)\Phi_{\mathbf{m}}(z)
&=\Delta(e-z)^{-\alpha}\Phi_{\mathbf{x}}\left(\left(e-\frac{1}{c}z\right)(e-z)^{-1}\right) \nonumber \\
&=\frac{1}{s_{\mathbf{x}}(1,\ldots,1)}\det\left((1-z_{\mu})^{-\alpha}\left(\frac{1-\frac{1}{c}z_{\mu}}{1-z_{\mu}}\right)^{x_{\nu}+r-\nu}\right) \nonumber \\
& \quad \cdot \frac{1}{V\left(\frac{1-\frac{1}{c}z_{1}}{1-z_{1}},\ldots,\frac{1-\frac{1}{c}z_{r}}{1-z_{r}}\right)}. \nonumber
\end{align}
Further, noticing that
$$
\frac{1-\frac{1}{c}z_{\mu}}{1-z_{\mu}}-\frac{1-\frac{1}{c}z_{\nu}}{1-z_{\nu}}
=\left(1-\frac{1}{c}\right)\frac{z_{\mu}-z_{\nu}}{(1-z_{\mu})(1-z_{\nu})},
$$
we obtain 
\begin{align}
\sum_{\mathbf{m} \in \mathscr{P}}d_{\mathbf{m}}\frac{(\alpha)_{\mathbf{m}}}{\left(r\right)_{\mathbf{m}}}M_{\mathbf{m}}(\mathbf{x};\alpha ,c)\Phi_{\mathbf{m}}(z)
&=\frac{\left(1-\frac{1}{c}\right)^{-\frac{r(r-1)}{2}}}{s_{\mathbf{x}}(1,\ldots,1)}
\frac{\det\left((1-z_{\mu})^{-(\alpha-r+1)}\left(\frac{1-\frac{1}{c}z_{\mu}}{1-z_{\mu}}\right)^{x_{\nu}+r-\nu}\right)}{V(z_{1},\ldots,z_{r})}. \nonumber
\end{align}
Here, by (\ref{eq:gen fnc of Meixner}), we remark 
$$
f_{\nu}(z_{\mu})=(1-z_{\mu})^{-(\alpha-r+1)}\left(\frac{1-\frac{1}{c}z_{\mu}}{1-z_{\mu}}\right)^{x_{\nu}+r-\nu}
=\sum_{m\geq 0}\frac{(\alpha -r+1)_{m}}{m!}M_{m}(x_{\nu}+r-\nu;\alpha -r+1,c)z_{\mu}.
$$
Therefore, we expand the above determinant expression in Schur function series by using Lemma\,\ref{thm:det lem}.
\begin{align}
\sum_{\mathbf{n} \in \mathscr{P}}d_{\mathbf{n}}\frac{(\alpha)_{\mathbf{n}}}{\left(r\right)_{\mathbf{n}}}M_{\mathbf{n}}(\mathbf{x};\alpha ,c)\Phi_{\mathbf{n}}(z)
&=\frac{1}{\delta !}\frac{\left(1-\frac{1}{c}\right)^{-\frac{r(r-1)}{2}}}{s_{\mathbf{x}}(1,\ldots,1)}\prod_{j=1}^{r}(\alpha -r+1)_{j-1} \nonumber \\
& \quad \cdot \sum_{\mathbf{m} \in \mathscr{P}}\frac{(\alpha)_{\mathbf{m}}}{(r)_{\mathbf{m}}}\det\left(M_{m_{\mu}+r-\mu}(x_{\nu}+r-\nu;\alpha -r+1,c)\right) \nonumber \\
& \quad \cdot s_{\mathbf{m}}(1,\ldots,1)\Phi_{\mathbf{m}}(z). \nonumber
\end{align}
Finally, by comparing of $\Phi_{\mathbf{m}}(z)$ on the above equation for $\mathbf{m} \in \mathscr{P}$, we obtain (\ref{eq:det Meixner}). 
Since both sides of (\ref{eq:det Meixner}) which are rational functions for $\alpha$ and $c$ hold for $\alpha \in \mathbb{C}\setminus \mathbb{Z}_{\leq r-1}, c\not=0$.

\noindent
{\rm{(2)}}\,From (\ref{eq:relation Meixner and Charlier}) and (\ref{eq:det Meixner}), we have
\begin{align}
C_{\mathbf{m}}(\mathbf{x};a)
&=\lim_{\alpha \to \infty}M_{\mathbf{m}}\left(\mathbf{x};\alpha ,\frac{a}{a+\alpha}\right) \nonumber \\
&=\frac{1}{\delta !}\frac{(-a)^{\frac{r(r-1)}{2}}}{s_{\mathbf{m}}(1,\ldots,1)s_{\mathbf{x}}(1,\ldots,1)} \nonumber \\
& \quad \cdot \lim_{\alpha \to \infty}\prod_{j=1}^{r}\frac{(\alpha -r+1)_{j-1}}{\alpha ^{j-1}}
\det\left(M_{m_{\mu}+r-\mu}\left(x_{\nu}+r-\nu;\alpha -r+1,\frac{a}{a+\alpha}\right)\right) \nonumber \\
&=\frac{1}{\delta !}\frac{(-a)^{\frac{r(r-1)}{2}}}{s_{\mathbf{m}}(1,\ldots,1)s_{\mathbf{x}}(1,\ldots,1)}
\det(C_{m_{\mu}+r-\mu}(x_{\nu}+r-\nu;a)). \nonumber 
\end{align}

\noindent
{\rm{(3)}}\,From (\ref{eq:relation Meixner and Krawtchouk}) and (\ref{eq:det Meixner}), we have
\begin{align}
K_{\mathbf{m}}(\mathbf{x};p ,N)
&=M_{\mathbf{m}}\left(\mathbf{x};-N ,\frac{p}{p-1}\right) \nonumber \\
&=\frac{1}{\delta !}\frac{p^{\frac{r(r-1)}{2}}}{s_{\mathbf{m}}(1,\ldots,1)s_{\mathbf{x}}(1,\ldots,1)}\prod_{j=1}^{r}(-N-r+1)_{j-1} \nonumber \\
& \quad \cdot \det\left(M_{m_{\mu}+r-\mu}\left(x_{\nu}+r-\nu;-N -r+1,\frac{p}{p-1}\right)\right). \nonumber
\end{align}
Here, by using (\ref{eq:relation Meixner and Krawtchouk}) again, 
$$
M_{m_{\mu}+r-\mu}\left(x_{\nu}+r-\nu;-N -r+1,\frac{p}{p-1}\right)=K_{m_{\mu}+r-\mu}(x_{\nu}+r-\nu;p,N+r-1).
$$
Therefore, we obtain the conclusion. 
\end{proof}

\section{Concluding remarks}
Interesting problems remain that are related to the multivariate Meixner, Charlier and Krawtchouk polynomials. 
First, we may consider a generalization of our discrete orthogonal polynomials 
for an arbitrary real value of multiplicity $d>0$. 
Actually, we can consider the multivariate Meixner, Charlier and Krawtchouk polynomials and their orthogonality without using analysis on the symmetric cones as follows. 

Let $n:=r+\frac{d}{2}r(r-1)$, $d>0$
\begin{align}
\label{eq:generalized d_{m}}
d_{\mathbf{m}}&:=
\prod_{j=1}^{r}\frac{\Gamma\left(\frac{d}{2}\right)}{\Gamma\left(\frac{d}{2}j\right)\Gamma\left(\frac{d}{2}(j-1)+1\right)} \nonumber \\
{} & \quad \cdot
\prod_{1\leq p<q\leq r}\left(m_{p}-m_{q}+\frac{d}{2}(q-p)\right)\frac{\Gamma\left( m_{p}-m_{q}+\frac{d}{2}(q-p+1)\right)}{\Gamma\left(m_{p}-m_{q}+\frac{d}{2}(q-p-1)+1\right)}. \nonumber \\
\Gamma_{\Omega}(\mathbf{s})&:=(2\pi)^{\frac{n-r}{2}}\prod_{j=1}^{r}\Gamma\left(s_{j}-\frac{d}{2}(j-1)\right), \nonumber \\
(\mathbf{s})_{\mathbf{k}}&:=\prod_{j=1}^{r}\left(s_{j}-\frac{d}{2}(j-1)\right)_{k_{j}}. \nonumber
\end{align}
Further, $P_{\mathbf{k}}^{(\frac{2}{d})}(\lambda_{1},\ldots,\lambda_{r})$ is an $r$-variable Jack polynomial 
and 
\begin{equation}
\label{eq:generalized spherical poly}
\Phi_{\mathbf{k}}^{(d)}(\lambda_{1},\ldots,\lambda_{r}):=\frac{P_{\mathbf{k}}^{(\frac{2}{d})}(\lambda_{1},\ldots,\lambda_{r})}{P_{\mathbf{k}}^{(\frac{2}{d})}(1,\ldots,1)}.
\end{equation} 
Furthermore, we introduce the generalized (Jack) binomial coefficients based on \cite{OO} by 
$$
\Phi_{\mathbf{m}}^{(d)}(1+\lambda_{1},\ldots,1+\lambda_{r})
=\sum_{\mathbf{k}\subset \mathbf{m}}\binom{\mathbf{m}}{\mathbf{k}}_{\frac{d}{2}}\Phi_{\mathbf{k}}^{(d)}(\lambda_{1},\ldots,\lambda_{r}).
$$
\begin{dfn}
We define the generalized multivariate Meixner, Charlier and Krawtchouk polynomials by 
\begin{align}
M_{\mathbf{m}}^{(d)}(\mathbf{x};\alpha ,c):=&\sum_{\mathbf{k}\subset \mathbf{m}}\frac{1}{d_{\mathbf{k}}}\frac{\left(\frac{n}{r}\right)_{\mathbf{k}}}{(\alpha)_{\mathbf{k}}}
\binom{\mathbf{m}}{\mathbf{k}}_{\frac{d}{2}}\binom{\mathbf{x}}{\mathbf{k}}_{\frac{d}{2}}\left(1-\frac{1}{c}\right)^{|\mathbf{k}|}, \\
C_{\mathbf{m}}^{(d)}(\mathbf{x};a):=&\sum_{\mathbf{k}\subset \mathbf{m}}\frac{1}{d_{\mathbf{k}}}\left(\frac{n}{r}\right)_{\mathbf{k}}
\binom{\mathbf{m}}{\mathbf{k}}_{\frac{d}{2}}\binom{\mathbf{x}}{\mathbf{k}}_{\frac{d}{2}}\left(-\frac{1}{a}\right)^{|\mathbf{k}|}, \\
K_{\mathbf{m}}^{(d)}(\mathbf{x};p ,N):=&\sum_{\mathbf{k}\subset \mathbf{m}}\frac{1}{d_{\mathbf{k}}}\frac{\left(\frac{n}{r}\right)_{\mathbf{k}}}{(-N)_{\mathbf{k}}}
\binom{\mathbf{m}}{\mathbf{k}}_{\frac{d}{2}}\binom{\mathbf{x}}{\mathbf{k}}_{\frac{d}{2}}\left(\frac{1}{p}\right)^{|\mathbf{k}|}\,\,\,\,(\mathbf{m} \subset N=(N,\ldots,N)).
\end{align}
\end{dfn}
By the definitions, 
Proposition\,\ref{thm:duality of MDOP} and \ref{thm:relation of MDOP} also hold for the generalized multivariate Meixner, Charlier and Krawtchouk polynomials. 
Therefore, we think the following conjecture is natural. 
\begin{conj}
Generating functions, orthogonality, difference equations and recurrence formulas also hold 
for the generalized multivariate Meixner, Charlier and Krawtchouk polynomials, 
as in Theorems\,\ref{thm:generating fnc of MDOP}, \ref{thm:Orthogonality of MDOP}, \ref{thm:Difference eq of MDOP} and \ref{thm:Recurrence formula of MDOP} respectively. 
Here, we consider $\Delta(e-z)=(1-z_{1})\cdots(1-z_{r})$.
\end{conj}
\noindent
We remark that when $d=1,2,4$ or $r=2, d \in \mathbb{Z}_{> 0}$ or $r=3, d=8$, this conjecture is proved by this paper and the classification of irreducible symmetric cones. 
However, it may be necessary to consider an algebraic treatment to prove the general case. 
In particular, since the difference equation for the multivariate Meixner polynomials is equivalent to the differential equation for the multivariate Laguerre polynomials which is explained by the degenerate double affine Hecke algebra \cite{Ka}, 
we expect the existence of a particular algebraic structure related to this algebra for our polynomials. 
Once we obtain such an interpretation, 
we may not only succeed in proving the above conjecture but also in providing further generalizations of our polynomials associated with root systems.

It is also valuable to give a group theoretic picture of our multivariate discrete orthogonal polynomials. 
In the one variable case, 
there are many geometric interpretations for these polynomials \cite{VK1}, \cite{VK2}.
Moreover, in the multivariate case for the Aomoto-Gelfand hypergeometric series, such group theoretic interpretations have recently been studied \cite{GVZ}, \cite{GMVZ}. 
On the other hand, since our multivariate discrete orthogonal polynomials have many rich properties which are generalizations of the one variable case, 
they are considered to be a good multivariate analogue of the Meixner, Charlier and Krawtchouk polynomials. 
Hence, for our multivariate discrete orthogonal polynomials, it seems that there are some group theoretic interpretations as some matrix elements or some spherical functions etc. 
We are also interested in a connection between our multivariate discrete orthogonal polynomials and the Aomoto-Gelfand type. 

We are interested in whether we can apply our method to other discrete orthogonal polynomials, for example, the Hahn polynomials which are special orthogonal polynomials in the Askey scheme \cite{KLS}, 
\begin{align}
Q_{m}(x;\alpha,\beta,N)&={_{3}F_2}\left(\begin{matrix}-m,m+\alpha +\beta +1,-x\\ \alpha+1,-N \end{matrix};1\right) \nonumber \\
&=\sum_{k=0}^{m}\frac{k!}{(-N)_{k}}\frac{(m+\alpha+\beta+1)_{k}}{(\alpha+1)_{k}}\binom{m}{k}\binom{x}{k}\,\,\,\,(m=0,1,\cdots,N). \nonumber
\end{align}
Namely, 
by considering ``some generating functions of the generating functions'' for these discrete orthogonal polynomials,  
we expect to obtain a correspondence between the Hahn polynomials and other orthogonal polynomials, for example, the Jacobi polynomials. 

Finally, we would like to raise the issue of applications of our multivariate Meixner, Charlier and Krawtchouk polynomials. 
The standard Meixner, Charlier and Krawtchouk polynomials of single discrete variable have found numerous applications
in combinatorics, stochastic processes, probability theory and mathematical physics (for their reference, see the introduction in \cite{GMVZ}). 
Hence, we hope that our multivariate polynomials can be applied to various situations 
and we intend to investigate these in research tasks in the future.



\bibliographystyle{amsplain}


\noindent Department of Pure and Applied Mathematics, 
Graduate School of Information Science and Technology, Osaka University, \\
1-1, Machikaneyama, Toyonaka, Osaka 560-0043, JAPAN.\\
E-mail: g-shibukawa@math.sci.osaka-u.ac.jp

\end{document}